\documentclass[a4paper,11pt,reqno,draft]{amsart}
\usepackage{amssymb,amsmath,array,hhline,amscd}

\usepackage[mathscr]{euscript}
\usepackage{stmaryrd}

\def\tbinom#1#2{{\left(\textstyle\genfrac{}{}{0pt}{}{#1}{#2}\right)}}

\renewcommand{\|}{\mathbin{\rm |}}
\renewcommand{\epsilon}{\varepsilon}
\renewcommand{\phi}{\varphi}
\renewcommand{\emptyset}{\varnothing}
\def\lm{\lambda}

\def\K{\mathbf K}
\def\Z{\mathbf Z}
\def\C{\mathbf C}
\def\O{\mathscr O}
\def\F{\mathcal F}

\def\L{\mathcal L}
\def\H{\mathcal H}
\def\N{{\rm Mat}}
\def\M{\mathcal M}
\def\R{\mathcal R}

\def\Hom{\mathop{\rm Hom}\nolimits}
\def\ch{\mathop{\rm ch}}

\def\ev{\mathop{\rm ev}\nolimits}
\def\r{\mathop{\rm r}\nolimits}
\def\h{\mathop{\rm h}\nolimits}
\def\Dist{\mathop{\rm Dist}\nolimits}

\def\UT{\mathop{\rm UT}\nolimits}

\def\chain{\mathop{\rm chain}}
\def\shape{\mathop{\rm shape}}

\newtheorem{theorem}{Theorem}
\newtheorem{proposition}[theorem]{Proposition}
\newtheorem{lemma}[theorem]{Lemma}
\newtheorem{corollary}[theorem]{Corollary}
\newtheorem{definition}[theorem]{Definition}

\newtheorem{conjecture}{Conjecture}

\renewcommand{\labelenumi}{{\rm \theenumi}}
\renewcommand{\theenumi}{{\rm(\roman{enumi})}}

\newenvironment{remark}{\refstepcounter{theorem}\par{\noindent\bf Remark \arabic{theorem}.}}{}

\def\sectsign{\mathhexbox278}

\renewcommand{\(}{\left(}
\renewcommand{\)}{\right)}

\renewcommand{\[}{\left[}
\renewcommand{\]}{\right]}

\def\<{\langle}
\def\>{\rangle}
\renewcommand{\le}{\leqslant}
\renewcommand{\ge}{\geqslant}
\def\={\equiv}

\def\U{\mathcal U}

\def\sgn{\mathop{\rm sgn}\nolimits}

\title[A local criterion for Weyl modules for groups of type A]{A local criterion for Weyl modules for groups of type A}
\author{Vladimir Shchigolev}

\address{
Department of Algebra\\
Faculty of Mathematics\\
Lomonosov Moscow State University\\
Leninskiye Gory, Moscow\\
119899, RUSSIA}

\email{shchigolev\_vladimir@yahoo.com}


\begin{document}

\maketitle

\begin{abstract} Let $G$ be a universal Chevalley group
over an algebraically closed field and
$\U^-$ be the subalgebra of $\Dist(G)$ generated
by all divided powers $X_{\alpha,m}$ with $\alpha<0$.
We conjecture an algorithm to determine if
$Fe^+_\omega\ne0$, where $F\in\U^-$, $\omega$ is a dominant weight
and $e^+_\omega$ is a highest weight vector of the Weyl module $\Delta(\omega)$. 
This algorithm does not use bases of $\Delta(\omega)$ and
is similar to the algorithm for irreducible modules
that involves stepwise raising the vector under investigation.
For an arbitrary $G$, this conjecture is proved in one direction
and for $G$ of type A in both.
\end{abstract}

\section{Introduction}\label{Introduction}


Let $\Sigma$ be a root system and $\K$ be an algebraically closed field.
Consider the semisimple complex
Lie algebra $\mathfrak L$ with Cartan subalgebra $\mathcal H$
and root system $\Sigma$. For each $\alpha\in\Sigma$,
we denote by $H_\alpha$ the element of $\H$ introduced
in~\cite[Lemma~1]{Steinberg_eng}.

We denote by $\Sigma^+$ a positive root system of $\Sigma$ and
by $\Pi=\{\alpha_1,\ldots,\alpha_\ell\}$ the simple root system
contained in $\Sigma^+$.
We choose elements $X_\alpha$, where $\alpha\in\Sigma$, so that
they together with $H_{\alpha_1},\ldots,H_{\alpha_\ell}$ form
a Chevalley basis of $\mathfrak L$ in the sense of~\cite[Theorem~1]{Steinberg_eng}.
Let $G$ be the universal Chevalley group
over $\K$ constructed by this
basis as in~\cite[\sectsign~3]{Steinberg_eng}.


Following~\cite[Theorem~2]{Steinberg_eng}, we denote by $\U_\Z$
the subring of the universal enveloping algebra of $\mathfrak L$
generated by $X_\alpha^m/m!$, where $\alpha\in\Sigma$ and
$m\in\Z^+$ (the set of nonnegative integers).
In this paper, we shall consider
modules over the algebra $\U:=\U_\Z\otimes_\Z\K$\label{hyperalgebra}, which is called
the {\it hyperalgebra} of $G$ (denoted by $\Dist(G)$ in~\cite{Jantzen2}). As a $\K$-algebra $\U$
is generated by the elements $X_{\alpha,m}:=(X_\alpha^m/m!)\otimes1_\K$.
Moreover, every rational $G$-module $V$ can be made into a $\U$-module
by the rule
$$
x_\alpha(t)v=\sum\nolimits_{m=0}^{+\infty}t^mX_{\alpha,m}v,
$$
where $v\in V$, $\alpha\in\Sigma$, $t\in\K$ and $x_\alpha(t)$ is the root
element of $G$ corresponding to $\alpha$ and $t$
(see~\cite[\sectsign~3]{Steinberg_eng}).
We shall also need the elements
$H_{\alpha,m}=\binom{H_\alpha}m\otimes1_\K$.
It is easy to show that these elements actually belong to $\U$
(e.g.~\cite[Corollary to Lemma~5]{Steinberg_eng}).
The reader should keep in mind the 
formula
$\binom{H_\alpha+r}m\otimes1_\K=\sum_{n=0}^mH_{\alpha,n}\binom r{m-n}$,
where $r\in\Z$.



\begin{proposition}\label{proposition:lcwman2:0}${}$
\!\!\!
\!\!
\!\!
\!\!
The products
$
\displaystyle
\prod\nolimits_{\alpha\in\Sigma^+}\!X_{-\alpha,m_{-\alpha}}
\cdot
\prod\nolimits_{i=1}^\ell H_{\alpha_i,n_i}
\cdot
\prod\nolimits_{\alpha\in\Sigma^+}\!X_{\alpha,m_\alpha},
$\\[4pt]
where $m_{-\alpha}$, $n_i$, $m_\alpha\in\Z^+$,
taken in some fixed order
form a basis of $\U$.
\end{proposition}
We shall always mean the grading of $\U$ in which $X_{\alpha,m}$
has weight $m\alpha$ and $H_{\alpha,m}$ has weight $0$.
Elements of $\U$ that are homogeneous with respect to this grading will be called {\it weight elements}.

We denote the $\K$-span of all the above basis elements
\begin{itemize}
\item with unitary second and third factors by $\U^-$;
\item with unitary first and third factors by $\U^0$;
\item with unitary third factor by $\U^{-,0}$.
\end{itemize}
It can be easily checked that each of these spaces
is a $\K$-algebra.
A vector $v$ of a rational $G$-module
is called {\it primitive} if $X_{\alpha,m}v=0$
for any $\alpha\in\Sigma^+$ and $m>0$ and
is called {\it simply primitive} if $X_{\alpha,1}v=0$
for any $\alpha\in\Sigma^+$.

To the Chevalley basis of $\mathfrak L$,
we associate the maximal torus $T$ in the standard way
(see~\cite[Lemma~21]{Steinberg_eng}, where the torus is denoted by $H$).
We denote by $X(T)$ and $X^+(T)$\label{XT} the set of weights and
the set of dominant (with respect to $\Sigma^+$)
weights of $T$ respectively. Let $\omega_1,\ldots,\omega_\ell$
be the fundamental
weights of $T$ corresponding to the roots $\alpha_1,\ldots,\alpha_\ell$.
Each weight $\omega\in X(T)$ has the form
$\omega=a_1\omega_1+\cdots+a_\ell\omega_\ell$, where $a_i\in\Z$.
Under this notation, we set $\h(\omega):=a_1+\cdots+a_\ell$\label{h}.
For each $\omega\in X^+(T)$, we denote by $\Delta(\omega)$
and $\nabla(\omega)$ the Weyl module and the co-Weyl module
with highest weight $\omega$ respectively. We also fix
some nonzero vector $e^+_\omega$ of $\Delta(\omega)$
having weight $\omega$. In what follows, we denote by
$X_{\alpha,m}^V$ the operator on a rational $G$-module $V$
acting as the left multiplication by $X_{\alpha,m}$.
We also denote by $V^\tau$ for $\tau\in X(T)$
the $\tau$-weight space of $V$. A {\it weigh decomposition} of an element $x$
belonging to $\U$ or to a rational $G$-module is
$x=x_1+\cdots+x_m$, where $x_1,\ldots,x_m$ are weight elements of
mutually distinct weights.

It is well known that a
vector of the irreducible rational
$G$-module $L(\omega)$\label{L(omega)} is nonzero if and only if it can be raised
up to proportionality to the highest weight vector $v^+_\omega$ of $L(\omega)$
by successive multiplications by the elements
$X_{\alpha,m}$ for $\alpha\in\Sigma^+$.
Thus for any $F\in\U^-$, we can determine whether $Fv^+_\omega\ne0$
by calculating directly in $\U$ (see Remark~\ref{remark:lcwman2:1}).
Note that we should at first consider the weight decomposition $F=F_1+\cdots+F_m$
and then investigate if $F_iv^+_\omega\ne0$ separately for each $i$.

Conjectures~\ref{conjecture:lcwman2:A} and~\ref{conjecture:lcwman2:B}
in Section~\ref{conjectures and reformulations} suggest
algorithms similar to this algorithm that concern
the Weyl module $\Delta(\omega)$ instead of
the irreducible module $L(\omega)$. These conjectures are proved
in one direction, see Corollary~\ref{corollary:lcwman:0}.
This result may be useful for establishing
nonzero homomorphism
between Weyl modules. Simple examples are given
in Sections~\ref{Example for GB2K and root 2alpha+beta}
and~\ref{Example for GB2K and root alpha+beta}.
In these examples, we choose $Y$ to be the corresponding
elements described at page~241 of~\cite{Carter_Lusztig}
evaluated at $\omega$. The coefficients in the commutator
relations for $B_2$ can be found in~\cite[\sectsign 25 Exercise~6]{Humphreys}.

Conjectures~\ref{conjecture:lcwman2:A} and~\ref{conjecture:lcwman2:B}
turn out to hold for $G=A_\ell(\K)$. This is proved in
Section~\ref{Criterion for A_ell}. We use there the standard basis
theorem for Weyl modules, the straightening rule and the results of
Section~\ref{Flows} concerning some properties of graphs of special type
called flows.

Section~\ref{Appendix: List of Notations} contains the list of notations
globally used throughout this paper.

{\bf Acknowledgments.} The author would like to thank Irina Suprunenko
for drawing his attention to this problem and a useful discussion.

\section{Conjectures and reformulations}\label{conjectures and reformulations}

\subsection{Conjectures}
Given a weight $\omega\in X(T)$, we define the $\K$-linear
map $\ev^\omega:\U^{-,0}\to\U^-$ by
$$
\ev^\omega
\biggl(
\prod\nolimits_{\alpha\in\Sigma^+}X_{-\alpha,m_{-\alpha}}
\cdot\displaystyle
\prod\nolimits_{i=1}^\ell H_{\alpha_i,n_i}
\biggr)=
\prod\nolimits_{\alpha\in\Sigma^+}X_{-\alpha,m_{-\alpha}}
\cdot\displaystyle
\prod\nolimits_{i=1}^\ell\tbinom{\omega(H_{\alpha_i})}{n_i}.
$$

The map $\ev^\omega$ enjoys the properties
{
\leftmargini=38pt
\renewcommand{\labelenumi}{{\bf\theenumi}}
\renewcommand{\theenumi}{(ev-\arabic{enumi})}
\begin{enumerate}
\item\label{property:ev-1}
      the restriction $\ev^\omega|_{\U^0}$ is a $\K$-algebra homomorphism from $\U^0$ to $\K$;\\[-5pt]
\item\label{property:ev-4}
      if $v$ is a vector of weight $\omega$ of a rational $G$-module
      and $P\in\U^{-,0}$, then $Pv=\ev^\omega(P)v$.
\end{enumerate}}

Given a weight $\omega\in X(T)$, a root $\alpha\in\Sigma^+$
and an integer $m\in\Z^+$, we define the operator $\r_{\alpha,m}^{\,\omega}$\label{r}
on $\U^-$ as follows. Let $F$ belong to $\U^-$.
By Proposition~\ref{proposition:lcwman2:0}, we have the unique
representation $X_{\alpha,m}F=P+E$, where $P\in\U^{-,0}$ and
$E$ belongs to the left ideal of $\U$ generated by
the elements $X_{\beta,k}$ with $\beta\in\Sigma^+$ and $k>0$.
Then we set $\r_{\alpha,m}^{\,\omega}(F):=\ev^\omega(P)$.
The map $\r_{\alpha,m}^{\,\omega}$ enjoys the properties
{
\leftmargini=32pt
\renewcommand{\labelenumi}{{\bf\theenumi}}
\renewcommand{\theenumi}{(r\,-\arabic{enumi})}
\begin{enumerate}
\item\label{property:r-1} if $F$ has weight $\tau$ then
                          $\r_{\alpha,m}^{\,\omega}(F)$
                          has weight $\tau+m\alpha$;
\item\label{property:r-2}
      if $v$ is a primitive vector of weight $\omega$ of a rational $G$-module,
      $F\in\U^-$, $\alpha\in\Sigma^+$ and $m\in\Z^+$,
      then $X_{\alpha,m}Fv=\r_{\alpha,m}^{\,\omega}(F)v$;

\end{enumerate}}

{\it Example.} We have $X_{\alpha_i,1}X_{-\alpha_i,2}=X_{-\alpha_i,1}(H_{\alpha_i,1}-1)+X_{-\alpha_i,2}X_{\alpha_i,1}$.
Therefore, $\r_{\alpha_i,1}^{\,\omega}(X_{-\alpha_i,2})=\ev^\omega(X_{-\alpha_i,1}(H_{\alpha_i,1}-1))=(a_i-1)X_{-\alpha_i,1}$,
where $\omega=a_1\omega_1+\cdots+a_\ell\omega_\ell$.

To formulate our conjectures, we consider the following
transformations
on $\U^-\times X^+(T)$:

{
\renewcommand{\labelenumi}{{\rm \theenumi}}
\renewcommand{\theenumi}{{\rm(\alph{enumi})}}
\begin{enumerate}
\item\label{transformation:a}
      $(F,\omega)\mapsto(\r_{\alpha,m}^{\,\omega}(F),\omega)$, where $\alpha\in\Sigma^+$ and
      $m\in\Z^+$;
\item\label{transformation:b}
      $(F,\omega)\mapsto (F,\omega-\delta)$, where
      $\delta$ is a
      weight of $X^+(T)$
      such that $\omega-\delta\in X^+(T)$.
\end{enumerate}}



\begin{definition}\label{definition:lcwman:0}
Let $F\in\U^-$ and $\omega\in X^+(T)$. The pair $(F,\omega)$
is called reducible {\rm(}simply reducible{\rm)} to a pair $(c,0)$, where $c\in\K^*$,
if there exists a sequence
\begin{equation}\label{equation:lcwman:-4}
(F^{(k)},\omega^{(k)}),\ldots,(F^{(0)},\omega^{(0)})
\end{equation}
of pairs of $\U^-\times X^+(T)$ such that
\begin{enumerate}
    \item\label{definition:lcwman:0:property:1} $(F^{(k)},\omega^{(k)})=(F,\omega)$ and $(F^{(0)},\omega^{(0)})=(c,0)$;
    \item\label{definition:lcwman:0:property:2} for any $i=0,{\ldots},k{-}1$, $(F^{(i)},\omega^{(i)})$ is derived from $(F^{(i+1)},\omega^{(i+1)})$
          by transformation~\ref{transformation:a} {\rm(}resp. transformation~\ref{transformation:a} with $m=1$ {\rm)} or transformation~\ref{transformation:b}.
\end{enumerate}

\end{definition}


\medskip

\begin{remark}\label{remark:lcwman2:0}
If $(F,\omega)$ is simply reducible to $(c,0)$, then $(F,\omega)$ is reducible to $(c,0)$.
\end{remark}

We shall also say that a pair $(F,\omega)$ is {\it reducible} ({\it simply reducible}) if this pair is
reducible (simply reducible) to some pair $(c,0)$, where $c\in\K^*$.

\begin{conjecture}\label{conjecture:lcwman2:A}
Let $F$ be a weight element of $\U^-$ and $\omega\in X^+(T)$.
We have $Fe^+_\omega\ne0$ if and only if the pair $(F,\omega)$ is reducible.
\end{conjecture}

\begin{conjecture}\label{conjecture:lcwman2:B}
Let $F$ be a weight element of $\U^-$ and $\omega\in X^+(T)$. We have $Fe^+_\omega\ne0$
if and only if the pair $(F,\omega)$
is simply reducible.
\end{conjecture}

\begin{remark}\label{remark:lcwman2:1}
Obviously, for a weight element $F$ of $\U^-$,
$Fv^+_\omega\ne0$ if and only if $F$ can be reduced to $c\in\K^*$
by the transformations $F\mapsto\r_{\alpha,m}^{\,\omega}(F)$.
\end{remark}

{\it Example.} The Weyl module $\Delta(0)$ is one-dimensional and therefore is irreducible.
Take $F=1+X_{-\alpha}$, where $\alpha\in\Sigma^+$. Then $Fe^+_0=e^+_0\ne0$, while
the pair $(F,0)$ is not reducible.

The above example shows why we consider only weight elements in
Conjectures~\ref{conjecture:lcwman2:A} and~\ref{conjecture:lcwman2:B}.
However, each of these conjectures if proved would allow us to answer whether
$Fe^+_\omega\ne0$ for an arbitrary $F\in\U^-$. Indeed, consider the weight
decomposition $F=F_1+\cdots+F_m$. Then $Fe^+_\omega\ne0$ if and only if
$F_ie^+_\omega\ne0$ for some $i$.

\subsection{Proof in one direction}
Of principal importance for the theory developed in this paper
is the following result.

\begin{lemma}\label{lemma:U^-}
Let $\omega,\delta\in X^+(T)$
such that $\omega-\delta\in X^+(T)$. Then there
exists the $\U^-$-homomorphism
$d^{\,\omega}_\delta:\Delta(\omega)\to\Delta(\omega-\delta)$
that takes $e^+_\omega$ to $e^+_{\omega-\delta}$.
\end{lemma}
\begin{proof} Let $w_0$ be the longest element of the Weyl group
of $\Sigma$. Suppose temporarily that ${\rm char}\,\K=0$.
Then $L(\delta)^*\cong L(-w_0\delta)$
by~\cite[Corollary~II.2.5 and Proposition~II.2.6]{Jantzen2}.
Since $v^+_{\omega-\delta}\otimes v^+_\delta$ is a nonzero primitive
vector of $L(\omega-\delta)\otimes L(\delta)$ having weigh $\omega$,
the universal property of Weyl modules~\cite[Lemma~II.2.13 b]{Jantzen2}
implies the existence of a nonzero homomorphism from $\Delta(\omega)$
to $L(\omega-\delta)\otimes L(\delta)$. Using $\Delta(\omega)\cong L(\omega)$
due to ${\rm char}\,\K=0$ and~\cite[Lemma~I.4.4]{Jantzen2}, we obtain
$$
\begin{array}{l}
K\cong\Hom_G(L(\omega),L(\omega-\delta)\otimes L(\delta))\cong
    \Hom_G(L(\omega)\otimes L(\delta)^*,L(\omega-\delta))\\[6pt]
    \cong\Hom_G(L(\omega)\otimes L(-w_0\delta),L(\omega-\delta)).
\end{array}
$$
Hence $L(\omega-\delta)$ is a composition factor of
$L(\omega){\otimes}L(-w_0\delta)$ with multiplicity~$1$.
Here we used that any rational $G$-module is semisimple.
Let $L(\tau)$ be a composition factor of
$L(\omega)\otimes L(-w_0\delta)$ with $\tau\ne\omega-\delta$.
Again using the above mentioned fact on semisimplicity, we get
$$
\begin{array}{l}
0\ne\Hom_G(L(\omega)\otimes L(-w_0\delta),L(\tau))\cong
    \Hom_G(L(\omega),L(\tau)\otimes L(-w_0\delta)^*)\\[6pt]
    =\Hom_G(L(\omega),L(\tau)\otimes L(\delta)),
\end{array}
$$
whence $\omega\le\tau+\delta$ and $\omega-\delta<\tau$.
In terms of characters, these facts can be written as
\begin{equation}\label{equation:lcwman:-3}
\ch L(\omega)\otimes L(-w_0\delta)=\ch L(\omega-\delta)+
\sum\nolimits_{\tau>\omega-\delta}a_\tau\ch L(\tau)
\end{equation}
for some $a_\tau\in\Z^+$.

Now let us return to the situation where $\K$ has
arbitrary characteristic. Then~(\ref{equation:lcwman:-3})
can be rewritten as
\begin{equation}\label{equation:lcwman:-2}
\ch\nabla(\omega)\otimes\nabla(-w_0\delta)=\ch \nabla(\omega-\delta)+
\sum\nolimits_{\tau>\omega-\delta}a_\tau\ch \nabla(\tau).
\end{equation}
The main result of~\cite{Mathieu2} implies that
$\nabla(\omega)\otimes \nabla(-w_0\delta)$ has a good filtration.
The factors of this filtration are given by~(\ref{equation:lcwman:-2}).
Therefore applying~\cite[II.4.16~Remark 4]{Jantzen2}, we obtain
that $\nabla(\omega-\delta)$ is the bottom factor
in some good filtration of $\nabla(\omega)\otimes\nabla(-w_0\delta)$.
The module $\Delta(\omega)\otimes\Delta(-w_0\delta)$
is contravariantly dual to $\nabla(\omega)\otimes\nabla(-w_0\delta)$,
whence it has $\Delta(\omega-\delta)$ as a top factor.
Let $\pi:\Delta(\omega)\otimes\Delta(-w_0\delta)\to\Delta(\omega-\delta)$
be the corresponding epimorphism of $G$-modules.

Let $u$ be a nonzero vector of $\Delta(-w_0\delta)$ of minimal
weight $-\delta$. We are going to prove that
$e^+_\omega\otimes u\notin\ker\pi$.
Suppose on the contrary that $e^+_\omega\otimes u\in\ker\pi$. The comparison
of weights shows that
$e^+_\omega\otimes(\bigoplus_{\tau>-\delta}\Delta(-w_0\delta)^\tau)\subset\ker\pi$.
Since $\dim\Delta(-w_0\delta)^{-\delta}=\dim\Delta(-w_0\delta)^{-w_0\delta}=1$,
we obtain $e^+_\omega{\otimes}\Delta(-w_0\delta)\subset\ker\pi$.

It is a well-known fact that $e^+_\omega\otimes\Delta(-w_0\delta)$
generates $\Delta(\omega)\otimes\Delta(-w_0\delta)$ as
a $G$-module. We sketch the proof for completeness.
Let $W$ denote the $G$-submodule of
$\Delta(\omega)\otimes\Delta(-w_0\delta)$ generated by
$e^+_\omega\otimes\Delta(-w_0\delta)$.
Suppose that $W\ne\Delta(\omega)\otimes\Delta(-w_0\delta)$
and take $\sigma_{\rm max}$ to be a maximal weight such that
$\Delta(\omega)^{\sigma_{\rm max}}\otimes\Delta(-w_0\delta)$
is not contained in $W$. Let $v_1\in\Delta(\omega)^{\sigma_{\rm max}}$
and $v_2\in\Delta(-w_0\delta)$ be arbitrary vectors.
We have $v_1=Fe^+_\omega$ for some $F\in\U^-$
having weight $\sigma_{\rm max}-\omega$. Obviously,
$W\ni F(e^+_\omega\otimes v_2)=v_1{\otimes}v_2+w$,
where $w\in\bigoplus_{\sigma>\sigma_{\rm max}}\!\Delta(\omega)^\sigma\otimes\Delta(-w_0\delta)$.
The latter sum is contained in $W$ by the choice of $\sigma_{\rm max}$,
whence $w\in W$ and $v_1\otimes v_2\in W$. We proved
$\Delta(\omega)^{\sigma_{\rm max}}\otimes\Delta(-w_0\delta)\subset W$,
which is a contradiction.

Hence $\ker\pi=\Delta(\gamma)\otimes\Delta(-w_0\delta)$,
since $\ker\pi$ is a $G$-submodule. This contradicts $\pi\ne0$.
We proved that $e^+_\omega\otimes u\notin\ker\pi$.
Hence $\pi(e^+_\omega\otimes u)=ce^+_{\omega-\delta}$, where $c\in\K^*$.
Now the required map is given by $d^{\,\omega}_\delta(v)=c^{-1}\pi(v\otimes u)$,
where $v\in\Delta(\omega)$.
\end{proof}

We note the following trivial property of these homomorphisms:
{
\leftmargini=32pt
\renewcommand{\labelenumi}{{\bf \theenumi}}
\renewcommand{\theenumi}{(d-\arabic{enumi})}
\begin{enumerate}
\item\label{property:d-1}
      if $\omega,\tau,\omega-\delta,\delta-\tau\in X^+(T)$
      then $d^{\,\omega}_\delta=d^{\,\omega-\tau}_{\delta-\tau}d^{\,\omega}_\tau$.
\end{enumerate}}

\begin{lemma}\label{lemma:lcwman:-1}
$\ker d^{\,\omega}_\delta$ is a $T$-module.
\end{lemma}
\begin{proof} Let $v\in\ker d^{\,\omega}_\delta$. Consider the weight decomposition
$v=v_1+\cdots+v_m$, where each $v_i\ne0$. It suffices to show that each
$v_i\in\ker d^{\,\omega}_\delta$.

There exist weight elements $F_1,\ldots,F_m$ of $\U^-$ such that
$v_i=F_ie^+_\omega$ for any $i=1,\ldots,m$. We have
$$
0=d^{\,\omega}_\delta(v)=d^{\,\omega}_\delta(v_1)+\cdots+d^{\,\omega}_\delta(v_m)=
F_1e^+_{\omega-\delta}+\cdots+F_me^+_{\omega-\delta}.
$$
Since the weights of $F_1,\ldots,F_m$ are mutually distinct, we obtain
$0\!=\!F_ie^+_{\omega-\delta}\!=\!d^{\,\omega}_\delta(v_i)$ for any $i=1,\ldots,m$.
\end{proof}

\begin{corollary}\label{corollary:lcwman:0}
Let $F{\in}\U^-$ and $\omega{\in}X^+(T)$. If $(F,\omega)$
is reducible then $Fe^+_\omega{\ne}0$.
\end{corollary}
\begin{proof} We apply induction on the length $k+1$ of sequence~(\ref{equation:lcwman:-4})
satisfying properties~\ref{definition:lcwman:0:property:1} and~\ref{definition:lcwman:0:property:2}
of Definition~\ref{definition:lcwman:0}.
For $k=0$, the result is obvious, since in that case $F=c\in\K^*$.

Now suppose that $k>0$. Then the inductive hypothesis yields
\begin{equation}\label{equation:lcwman:-1.5}
F^{(k{-}1)}e^+_{\omega^{(k{-}1)}}\linebreak\ne0,
\end{equation}
since the sequence $(F^{(k-1)},\omega^{(k-1)}),\ldots,(c,0)$ has length $k$ and we assume that for sequences of length
smaller that $k+1$ the result is true.

{\it Case 1: $(F^{(k-1)},\omega^{(k-1)})$ is derived from
             $(F,\omega)$ by
             transformation~\ref{transformation:a}.}
In that case, we have $\omega^{(k-1)}=\omega$ and
$F^{(k-1)}=\r_{\alpha,m}^{\,\omega}(F)$ for some
$\alpha\in\Sigma^+$ and
$m\in\Z^+$.
By~\ref{property:r-2} and
(\ref{equation:lcwman:-1.5}) we obtain
$$
X_{\alpha,m}Fe^+_\omega=
\r_{\alpha,m}^{\,\omega}(F)e^+_\omega=
F^{(k-1)}e^+_{\omega^{(k-1)}}\ne0,
$$
whence $Fe^+_\omega\ne0$.

{\it Case 2: $(F^{(k-1)},\omega^{(k-1)})$ is derived from
             $(F,\omega)$ by transformation~\ref{transformation:b}.}
In that case, we have $F^{(k-1)}=F$ and
$\omega^{(k-1)}=\omega-\delta$ for some $\delta\in X^+(T)$.
By Lemma~\ref{lemma:U^-} and
(\ref{equation:lcwman:-1.5}) we obtain
$$
d^{\,\omega}_\delta(Fe^+_\omega)=
Fd^{\,\omega}_\delta\(e^+_\omega\)=Fe^+_{\omega-\delta}
=F^{(k-1)}e^+_{\omega^{(k-1)}}\ne0,
$$
whence again $Fe^+_\omega\ne0$.
\end{proof}

It follows from this corollary and Remark~\ref{remark:lcwman2:0}
that Conjecture~\ref{conjecture:lcwman2:B} implies Conjecture~\ref{conjecture:lcwman2:A}
as follows:
$$
\begin{array}{l}
(F,\omega)\mbox{ is reducible}\stackrel{\text{Cor.\ref{corollary:lcwman:0}}}{\Longrightarrow}Fe^+_\omega\ne0,\\[6pt]
Fe^+\ne0\stackrel{\text{Con.\ref{conjecture:lcwman2:B}}}{\Longrightarrow}(F,\omega)\mbox{ is simply reducible}\stackrel{\text{Rem.\ref{remark:lcwman2:0}}}{\Longrightarrow}(F,\omega)\mbox{ is reducible}.
\end{array}
$$

\subsection{Reformulations} The following result will be used in
Section~\ref{Criterion for A_ell}.

\begin{lemma}\label{lemma:reformulation:B}
Conjecture~\ref{conjecture:lcwman2:B} holds for all $(F,\omega)$ if and only if
for any nonzero $\omega\in X^+(T)$, the intersection
\begin{equation}\label{equation:lcwman:-1}
\bigcap\left\{\ker d^{\,\omega}_{\omega_i}\|i=1,\ldots,\ell\mbox{ and }\omega-\omega_i\in X^+(T)\right\}
\end{equation}
does not contain
nonzero simply primitive vectors.
\end{lemma}
\begin{proof}
{\it ``Only if'' part.} Suppose that $\omega$ is a nonzero dominant weight
and~(\ref{equation:lcwman:-1}) contains a nonzero
simply primitive vector $v$. By Lemma~\ref{lemma:lcwman:-1}, we cam assume that $v$ is a weight vector.
We obviously have $v=Fe^+_\omega$ for some weight element $F$ of $\U^-$.
We claim that $(F,\omega)$ is not simply reducible, i.e. Conjecture~\ref{conjecture:lcwman2:B}
is violated for $(F,\omega)$.

Suppose, on the contrary, that $(F,\omega)$ is simply reducible to $(c,0)$, where $c\in\K^*$.
Then there exists a sequence of form~(\ref{equation:lcwman:-4})
satisfying properties~\ref{definition:lcwman:0:property:1} and~\ref{definition:lcwman:0:property:2} of Definition~\ref{definition:lcwman:0},
where in~\ref{definition:lcwman:0:property:2} $m$ always equals $1$ if transformation~\ref{transformation:a}
is applied. Without loss of generality we may suppose that
in each transformation~\ref{transformation:b} used to derive $(F^{(i)},\omega^{(i)})$ from $(F^{(i+1)},\omega^{(i+1)})$,
we take $\delta\ne0$. We have $k>0$, since otherwise $F=c$ and $v\notin\ker d^{\,\omega}_{\omega_i}$
for any $i=1,\ldots,\ell$ such that $\omega-\omega_i\in X^+(T)$.
Note that at least one such $i$ exists since $\omega\ne0$.
Hence $v$ does not belong to~(\ref{equation:lcwman:-1}).

Like any pair of sequence~(\ref{equation:lcwman:-4}), the pair $(F^{(k-1)},\omega^{(k-1)})$  is reducible.
Thus Corollary~\ref{corollary:lcwman:0} implies
\begin{equation}\label{equation:lcwman:-0.75}
F^{(k-1)}e^+_{\omega^{(k-1)}}\ne0.
\end{equation}

{\it Case 1: $(F^{(k-1)},\omega^{(k-1)})$
is derived from $(F,\omega)$ by transformation~\ref{transformation:a} with $m=1$.}
In that case, we have $\omega^{(k-1)}=\omega$ and
$F^{(k-1)}=\r_{\alpha,1}^{\,\omega}(F)$ for some $\alpha\in\Sigma^+$.
Applying the fact that $v$ is simply primitive and~\ref{property:r-2},
we obtain
$$
0=X_{\alpha,1}v=X_{\alpha,1}Fe^+_\omega=\r_{\alpha,1}^{\,\omega}(F)e^+_\omega=F^{(k-1)}e^+_{\omega^{(k-1)}},
$$
which contradicts~(\ref{equation:lcwman:-0.75}).

{\it Case 2: $(F^{(k-1)},\omega^{(k-1)})$ is derived from $(F,\omega)$ by
             transformation~\ref{transformation:b}.}
In that case, we have $F^{(k-1)}=F$ and $\omega^{(k-1)}=\omega-\delta$
for some nonzero weight $\delta$ of $X^+(T)$.
Since we supposed that $\delta\ne0$, there exists $j=1,\ldots,\ell$ such that
$\delta-\omega_j\in X^+(T)$. Since $\omega^{(k-1)}\in X^+(T)$,
we clearly have $\omega-\omega_j=\omega^{(k-1)}+\delta-\omega_j\in X^+(T)$.
Thus $v\in\ker d^{\,\omega}_{\omega_j}$, since $v$ belongs
to~(\ref{equation:lcwman:-1}). By~\ref{property:d-1}, we obtain
$$
0=d^{\,\omega-\omega_j}_{\delta-\omega_j}d^{\,\omega}_{\omega_j}(v)=
d^{\,\omega}_\delta(v)=F d^{\,\omega}_\delta(e^+_\omega)=
Fe^+_{\omega-\delta}=F^{(k-1)}e^+_{\omega^{(k-1)}},
$$
which again contradicts~(\ref{equation:lcwman:-0.75}).

{\it ``If'' part.} Suppose that for
any nonzero dominant weight $\omega$, intersection~(\ref{equation:lcwman:-1})
does not contain nonzero simply primitive vectors.
In view of Corollary~\ref{corollary:lcwman:0}, we must only prove
that for any $F\in\U^-$ having weight $\tau$ and $\omega\in X^+(T)$
such that $Fe^+_\omega\ne0$, the pair $(F,\omega)$ is simply reducible.

We apply induction on $(-\tau,\h(\omega))$ with
componentwise order on such pairs.
%
%
If $\omega=0$ then $Fe^+_\omega\ne0$ implies that $\tau=0$ and
$F$ is a nonzero element of $\K$. In that case $(F,\omega)$
is simply reducible (to itself).
Therefore, we suppose
that $\omega\ne0$.


{\it Case 1: $Fe^+_\omega$ is not simply primitive.}
We have $X_{\alpha,1}Fe^+_\omega\ne0$ for some $\alpha\in\Sigma^+$.
By~\ref{property:r-2}, we obtain
$$
0\ne X_{\alpha,1}Fe^+_\omega=\r_{\alpha,1}^{\,\omega}(F)e^+_\omega.
$$
By~\ref{property:r-1}, the vector $\r_{\alpha,1}^{\,\omega}(F)$
has weight $\tau+\alpha\le 0$. Since
$(-(\tau+\alpha),\h(\omega))$ is smaller than
$(-\tau,\h(\omega))$, the inductive hypothesis implies
that $(\r_{\alpha,1}^{\,\omega}(F),\omega)$
is simply reducible.
It remains to notice
that $(\r_{\alpha,1}^{\,\omega}(F),\omega)$ is derived from $(F,\omega)$
by transformation~\ref{transformation:a} with $m=1$.

{\it Case 2: $Fe^+_\omega$ does not belong to~{\rm(}\ref{equation:lcwman:-1}{\rm)}.}
There exists some $j=1,\ldots,\ell$ such that $\omega-\omega_j\in X^+(T)$
and $Fe^+_\omega\notin\ker d^{\,\omega}_{\omega_j}$. We have
$$
0\ne d^{\,\omega}_{\omega_j}(Fe^+_\omega)=
Fd^{\,\omega}_{\omega_j}(e^+_\omega)=Fe^+_{\omega-\omega_j}.
$$
Since $\h(\omega-\omega_j)=\h(\omega)-1$, the pair
$(-\tau,\h(\omega-\omega_j))$ is smaller than
the pair $(-\tau,\h(\omega))$ and the inductive hypothesis implies
that $(F,\omega-\omega_j)$
is simply reducible.
It remains to notice that $(F,\omega-\omega_j)$ is derived from $(F,\omega)$
by transformation~\ref{transformation:b}.

It is natural to consider the remaining case when we are in neither Case~1 nor Case~2.
Then $Fe^+_\omega$ is a nonzero simply primitive vector belonging to~(\ref{equation:lcwman:-1}).
Since $\omega\ne0$, this contradicts the assumption of this part of the proof.
\end{proof}

Arguing as above but dropping the word ``simply'' and the condition $m{=}1$,
we obtain the following result.

\begin{lemma}\label{lemma:reformulation:A}
Conjecture~\ref{conjecture:lcwman2:A} holds if and only if
for any nonzero $\omega\in X^+(T)$,
intersection~{\rm(}\ref{equation:lcwman:-1}{\rm)} does not contain
nonzero primitive vectors {\rm(}equivalently, nonzero $G$-submodules{\rm)}.
\end{lemma}

\subsection{Twisting operators}\label{twisting operators}
Now we shift attention to the hyperalgebra
by introducing for any $\delta{\in}X(T)$ the {\it twisting operator}
$\theta_\delta$\label{theta} on $\U^{-,0}$ as the $\K$-linear map defined by
$$
{\arraycolsep=1pt
\begin{array}{rcl}
\displaystyle
\theta_\delta
\biggl(
\prod\nolimits_{\alpha\in\Sigma^+}X_{-\alpha,m_{-\alpha}}
&\cdot&\displaystyle
\prod\nolimits_{i=1}^\ell H_{\alpha_i,n_i}
\biggr)=\\[12pt]
\displaystyle
\prod\nolimits_{\alpha\in\Sigma^+}X_{-\alpha,m_{-\alpha}}
&\cdot&\displaystyle
 \prod\nolimits_{i=1}^\ell\tbinom{H_{\alpha_i}+\delta(H_{\alpha_i})}{n_i}\otimes1_\K.
\end{array}}
$$
The map $\theta_\delta$ enjoys the properties
{
\leftmargini=31pt
\renewcommand{\labelenumi}{{\bf\theenumi}}
\renewcommand{\theenumi}{($\theta$-\arabic{enumi})}
\begin{enumerate}
\item\label{property:theta-1}
      $\theta_\delta$ is a $\K$-algebra automorphism of $\U^{-,0}$;\\[-5pt]
%
%
%
\item\label{property:theta-3}
      if $\omega,\delta,\omega-\delta\in X^+(T)$, $P\in\U^{-,0}$ and
      $v\in\Delta(\omega)$, then \\[3pt]
      $d^{\,\omega}_\delta(Pv)=\theta_\delta(P)d^{\,\omega}_\delta(v)$.
\end{enumerate}}
We need to comment only on~\ref{property:theta-3}. By linearity,
it is enough to consider only the case $P=F'H$,
where $F'\in\U^-$ and $H\in\U^0$, and $v=Fe^+_\omega$,
where $F$ is an element of $\U^-$ having weight $\tau$.
Applying~\ref{property:ev-4},~\ref{property:theta-1}, the definition of
$\theta_\delta$ and the fact that
$d^{\,\omega}_\delta$ is a $\U^-$-homomorphism, we obtain
$$
\begin{array}{rl}
&d^{\,\omega}_\delta(Pv)=F'd^{\,\omega}_\delta(Hv)=
F'd^{\,\omega}_\delta\(\ev^{\omega+\tau}(H)Fe^+_\omega\)\\[6pt]
&=F'\ev^{\omega+\tau}(H)Fd^{\,\omega}_\delta\(e^+_\omega\)=
\ev^{\omega+\tau}(H)F'Fe^+_{\omega-\delta},\\[6pt]
&\theta_\delta(P)d^{\,\omega}_\delta(v)=F'\theta_\delta(H)Fd^{\,\omega}_\delta(e^+_\omega)=
F'\theta_\delta(H)Fe^+_{\omega-\delta}\\[6pt]
&=\ev^{\omega-\delta+\tau}(\theta_\delta(H))F'Fe^+_{\omega-\delta}.
\end{array}
$$
Hence it remains to prove that
$\ev^{\omega-\delta+\tau}(\theta_\delta\,(H))=\ev^{\omega+\tau}(H)$.
By~\ref{property:ev-1} and~\ref{property:theta-1},
it is enough to consider the case $H=H_{\alpha_i,n}$.
We have
$$
\begin{array}{l}
\displaystyle
\ev^{\omega-\delta+\tau}\(\theta_\delta\(H_{\alpha_i,n}\)\)=
\ev^{\omega-\delta+\tau}\(\tbinom{H_{\alpha_i}+\delta(H_{\alpha_i})}n\otimes1_\K\)=\\[12pt]
\tbinom{(\omega-\delta+\tau)(H_{\alpha_i})+\delta(H_{\alpha_i})}n\otimes1_\K
=\tbinom{(\omega+\tau)(H_{\alpha_i})}n\otimes1_\K=\ev^{\omega+\tau}(H_{\alpha_i,n}).
\end{array}
$$

For any $\alpha\in\Sigma^+$ and $m\in\Z^+$, we define the $\K$-linear
operator $\eta_{\alpha,m}$\label{eta} on $\U^{-,0}$ as follows.
Let $F$ belong to $\U^-$. Then $X_{\alpha,m}F=P+E$,
where $P\in\U^{-,0}$ and $E$ belongs to the left ideal of $\U$
generated by the elements $X_{\beta,k}$
with $\beta\in\Sigma^+$ and $k>0$.
We set $\eta_{\alpha,m}(F):=P$. The last element already
was introduced, when we defined $\r_{\alpha,m}^{\,\omega}$.
We clearly have $\ev^\omega(\eta_{\alpha,m}(F))=\r_{\alpha,m}^{\,\omega}(F)$.
The operator $\eta_{\alpha,m}$ enjoys the property
{
\leftmargini=33pt
\renewcommand{\labelenumi}{{\bf\theenumi}}
\renewcommand{\theenumi}{($\eta$\,-\arabic{enumi})}
\begin{enumerate}
\item\label{property:eta-1}
      if $v$ is a primitive vector of a rational $G$-module,
      $F\in\U^-$, $\alpha\in\Sigma^+$ and $m\in\Z^+$,
      then $X_{\alpha,m}Fv=\eta_{\alpha,m}(F)v$.
\end{enumerate}}

\renewcommand\vert{\vphantom{{A^A}^A}}

\begin{lemma}\label{lemma:lcwman:0}
Let $\Phi(x_1,\ldots,x_k,y)$ be an element of the free
associative algebra over $\K$ with generators $x_1,\ldots,x_k,y$
{\rm(}noncommutative associative polynomial\,{\rm)} linear in $y$ and
$\omega,\delta,\omega{-}\delta\in X^+(T)$.
Let $D^{\,\omega}_\delta$ be any operator on the module
$V{:=}\Delta(\omega)\oplus\Delta(\omega{-}\delta)$ whose restriction to
$\Delta(\omega)$ coincides with $d^{\,\omega}_\delta$. We have
$$
\Phi\(X_{\beta_1,m_1}^V,\ldots,X_{\beta_k,m_k}^V,D^{\,\omega}_\delta\)(Fe^+_\omega)=
\Phi\(\eta_{\vert\beta_1,m_1},\ldots,\eta_{\vert\beta_k,m_k},\theta_\delta\)(F)\cdot e^+_{\omega-\delta}
$$
for any $\beta_1,\ldots,\beta_k\in\Sigma^+$, $m_1,\ldots,m_k\in\Z^+$ and $F\in\U^-$.
\end{lemma}
\begin{proof} By linearity, it suffices to prove that
\begin{equation}\label{equation:lcwman:-0.5}
\begin{array}{l}
\displaystyle
X_{\beta_1,m_1}^V\cdots
X_{\beta_q,m_q}^V\,
D^{\,\omega}_\delta\,
X_{\beta_{q+1},m_{q+1}}^V\cdots
X_{\beta_k,m_k}^V
(Fe^+_\omega)\\[6pt]
\displaystyle
=\eta_{\vert\beta_1,m_1}\cdots\eta_{\vert\beta_q,m_q}\,\theta_\delta\,
 \eta_{\vert\beta_{q+1},m_{q+1}}\cdots\eta_{\vert\beta_k,m_k}(F)\cdot
 e^+_{\omega-\delta}.
\end{array}
\end{equation}
By~\ref{property:eta-1} and~\ref{property:theta-3}, the left-hand side of~(\ref{equation:lcwman:-0.5}) equals
$$
\begin{array}{l}
\displaystyle
X_{\beta_1,m_1}\cdots X_{\beta_q,m_q}d^{\,\omega}_\delta(
X_{\beta_{q+1},m_{q+1}}\cdots
X_{\beta_k,m_k}Fe^+_\omega
)\\[6pt]
\displaystyle
=X_{\beta_1,m_1}\cdots X_{\beta_q,m_q}d^{\,\omega}_\delta
\(
\eta_{\vert\beta_{q+1},m_{q+1}}\cdots
\eta_{\vert\beta_k,m_k}(F)\cdot e^+_\omega
\)\\[6pt]
\displaystyle
=X_{\beta_1,m_1}\cdots X_{\beta_q,m_q}
\(\theta_\delta\,\eta_{\vert\beta_{q+1},m_{q+1}}\cdots\eta_{\vert\beta_k,m_k}(F)\)\cdot e^+_{\omega-\delta}\\[6pt]
\displaystyle
=\eta_{\vert\beta_1,m_1}\cdots\eta_{\vert\beta_q,m_q}\,\theta_\delta\,
 \eta_{\vert\beta_{q+1},m_{q+1}}\cdots\eta_{\vert\beta_k,m_k}(F)\cdot
 e^+_{\omega-\delta},
\end{array}
$$
which is exactly the right-hand side of~(\ref{equation:lcwman:-0.5}).
\end{proof}

\subsection{Example for $G=B_2(\K)$ and root $2\alpha+\beta$}\label{Example for GB2K and root 2alpha+beta}
We are going to show a possible
application of Corollary~\ref{corollary:lcwman:0}. Let $\mathop{\rm char}\K=p>2$,
$G=B_2(\K)$ and $\Sigma^+=\{\alpha,\beta\}$, where $\alpha$
is short and $\beta$ is long. We denote by $\omega_\alpha$ and
$\omega_\beta$ the fundamental weights corresponding
to $\alpha$ and $\beta$ respectively. Take a dominant weight
$\omega=a\omega_\alpha+b\omega_\beta$ such that $a\ge2$, $b\ge1$ and
$a+b+1\=0\pmod p$.
We claim that $\Hom_G(\Delta(\omega-(2\alpha+\beta)),\Delta(\omega))\ne0$.
Indeed, consider the elements
$$
{\arraycolsep=0pt
\begin{array}{l}
Y=X_{-\alpha,1}^2X_{-\beta,1}-N_{\alpha,\beta}\,b\,X_{-\alpha,1}X_{-\alpha-\beta,1}+b(b+1)\tfrac{N_{\alpha,\beta}N_{\alpha,\alpha+\beta}}2X_{-2\alpha-\beta,1},\\[6pt]
Z=2X_{-\alpha,1}X_{-\beta,1}-N_{\alpha,\beta}\,b\,X_{-\alpha-\beta,1},
\end{array}}
$$
where $[X_\delta,X_\gamma]=N_{\delta,\gamma}X_{\delta+\gamma}$
for $\delta,\gamma,\delta+\gamma\in\Sigma$.
Elementary calculations show
$$
\r_{\alpha,1}^\omega(Y){=}(a+b+1)Z,\quad \r_{\alpha,2}^\omega(Y){=}(a+b+2)(a+b+1)X_{-\beta,1},\quad \r_{\beta,1}^\omega(Y){=}0.
$$
The equivalence $a+b+1\=0\pmod p$ and~\ref{property:r-2} imply that
$Ye^+_\omega$ is a primitive vector of $\Delta(\omega)$ having weight
$\omega-(2\alpha+\beta)$.
%
We are going to show that this vector is nonzero by showing that the pair
$(Y,\omega)$ is reducible.
Choose any integers $a'$ and $b'$ such that $0\le a'\le a$ and $0\le b'\le b$ and
set $\omega':=a'\omega_\alpha+b'\omega_\beta$.
%
Consider the following sequence of transformations:
$$
{
\arraycolsep=0pt
\begin{array}{l}
\begin{CD}
(Y,\omega)
@>\delta=\omega'>\text{transf.~\ref{transformation:b}}>
(Y,\omega{-}\omega')
@>\r_{\alpha,1}^{\omega-\omega'}>\text{transf.~\ref{transformation:a}}>
((a{-}a'{+}b{+}1)Z,\omega{-}\omega')
@>\r_{\beta,1}^{\omega-\omega'}>\text{transf.~\ref{transformation:a}}>
\end{CD}\\[16pt]
\begin{CD}
(a'(2b'{-}b)X_{-\alpha,1},\omega{-}\omega')
@>\r_{\alpha,1}^{\omega-\omega'}>\text{transf.~\ref{transformation:a}}>
(a'(2b'{-}b)(a{-}a')\,1_\K,\omega{-}\omega')
@>\delta=\omega-\omega'>\text{transf.~\ref{transformation:b}}>
\end{CD}\\[20pt]
\begin{CD}
(a'(2b'{-}b)(a{-}a')\,1_\K,0).
\end{CD}
\end{array}}
$$
\vspace{2pt}

\noindent
To ensure that $a'(2b'-b)(a-a')\not\equiv0\pmod p$, we set
$a':=1$ if \linebreak $a\not\equiv1\pmod p$,
$a':=2$ if $a\equiv1\pmod p$,
$b':=0$ if $b\not\equiv0\pmod p$
and $b':=1$ if $b\equiv0\pmod p$.

\subsection{Example for $G=B_2(\K)$ and root $\alpha+\beta$}\label{Example for GB2K and root alpha+beta}
In this example, we use the same notation as in the previous one
but assume that $\mathop{\rm char}\K$ is an arbitrary prime $p$.
Take a dominant weight $\omega=a\omega_\alpha+b\omega_\beta$ such that $a\ge1$,
$b\ge1$ and $a+2b+2\=0\pmod p$. We claim that
$\Hom_G(\Delta(\omega-(\alpha+\beta)),\Delta(\omega))\ne0$.
Indeed, consider the element
$$
Y=X_{-\alpha,1}X_{-\beta,1}-N_{\alpha,\beta}\,b\,X_{-\alpha-\beta,1}.
$$
Elementary calculations show
$$
\r_{\alpha,1}^\omega(Y)=(a+2b+2)X_{-\beta,1},\quad \r_{\beta,1}^\omega(Y)=0.
$$
The equivalence $a+2b+2\=0\pmod p$ and~\ref{property:r-2} imply that
$Ye^+_\omega$ is a primitive vector of $\Delta(\omega)$ having weight
$\omega-(\alpha+\beta)$. We are going to show that this vector is nonzero
by showing that the pair $(Y,\omega)$ is reducible.
Choose any integers $a'$ and $b'$ such that $0\le a'\le a$ and $0\le b'\le b$ and
set $\omega':=a'\omega_\alpha+b'\omega_\beta$.
Consider the following sequence of transformations:
$$
{
\arraycolsep=0pt
\begin{array}{l}
\begin{CD}
(Y,\omega)\!
@>\delta=\omega'>\text{transf.~\ref{transformation:b}}>
\!
(Y,\omega{-}\omega')
@>\r_{\alpha,1}^{\omega-\omega'}>\text{transf.~\ref{transformation:a}}>
\!
((a{-}a'{+}2b{+}2)X_{-\beta,1},\omega{-}\omega')
\!
@>\r_{\beta,1}^{\omega-\omega'}>\text{transf.~\ref{transformation:a}}>\!\!\!
\end{CD}\\[16pt]
\begin{CD}
(a'(b'{-}b)\,1_\K,\omega{-}\omega')
@>\delta=\omega-\omega'>\text{transf.~\ref{transformation:b}}>
(a'(b'{-}b)\,1_\K,0).
\end{CD}
\end{array}}
$$
To ensure that $a'(b'-b)\not\equiv0\pmod p$, we set
$a':=1$ and $b':=0$ if $b\not\equiv0\pmod p$ and
$b':=1$ if $b\equiv0\pmod p$.

\section{Criterion for $G=A_\ell(\K)$.}\label{Criterion for A_ell}

\subsection{Flows}\label{Flows}
In what follows, we call $s$ the {\it beginning} and $t$ the {\it end}
of $(s,t)$ or say that $(s,t)$ {\it begins} at $s$ and {\it ends} at $t$.

\begin{definition}\label{definition:lcwman:1}
A directed graph $\Gamma$ is called a flow with sources $a_1,\ldots,a_q$
and sinks $b_1,\ldots,b_q$ if
{
\begin{enumerate}
\itemsep=3pt
\item the set of vertices of $\Gamma$ is $\Z$; every edge of $\Gamma$ has the form $(s,t)$ with $s<t$;
      the number of edges of $\Gamma$ is finite;
\item for each $j=1,\ldots,q$, there is exactly one edge of $\Gamma$
      beginning at $a_j$;
\item for each $j=1,\ldots,q$, there is exactly one edge of $\Gamma$
      ending at $b_j$;
\item for each integer $r$ distinct from $a_1,\ldots,a_q,b_1,\ldots,b_q$,
       either no edge of $\Gamma$ ends or begins at $r$
       or exactly one edge of $\Gamma$ ends at $r$ and
       exactly one edge of $\Gamma$ begins at $r$.
       In the latter case, $r$ is called a transit point
       of $\Gamma$.
\end{enumerate}}
\end{definition}
\noindent
For $q=0$, the only possible flow is the {\it empty flow} $\Gamma_\emptyset$, which has no edges.

We denote by $S_q$ the symmetric group thought of as the group of all bijections
of $\{1,\ldots,q\}$. We assume that the elements of $S_q$ act on this set
on the left and multiply them accordingly. We also identify $S_q$
with a subgroup of $S_{q+1}$ assuming that $\sigma(q+1)=q+1$ for any $\sigma\in S_q$.

We denote the set of all flows $\Gamma$ with sources $a_1,\ldots,a_q\le i$
and sinks $b_1,{\ldots},b_q{>}i$ having no transit point greater than $i$ by
$\F_i(a_1,{\ldots},a_q;b_1,{\ldots},b_q)$\label{F}. Clearly, $\Gamma$ decomposes
into connected components. If there is a component of $\Gamma$
containing $a_j$ and $b_k$,
then we say that source $a_j$ and sink $b_k$
are linked in $\Gamma$. Throughout this section, we suppose
for definiteness that $a_1<a_2<\cdots<a_q$ and
$b_q<b_{q-1}<\cdots<b_1$. Let $\sigma$ denote the permutation
in $S_q$ such that
$a_j$ and $b_{{\sigma(j)}}$ are linked in $\Gamma$.
We shall occasionally call $\sigma$ the {\it linking permutation} of
$\Gamma$.
We define the {\it $i$-sign of $\Gamma$} to be
\begin{equation}\label{equation:lcwman:0}
\sgn_i(\Gamma):=\sgn\sigma\cdot(-1)^{\sum_{j=1}^q(b_j-i)+\text{number of transit points of }\Gamma}.
\end{equation}

{\it Example.} Let $\Gamma$ have the following edges
$(5,9)$, $(3,8)$, $(6,7)$, $(4,6)$, $(2,4)$, $(1,3)$.
Then
$\Gamma\in\F_6(1,2,5;9,8,7)$
and has transit points $3,4,6$.
Its linking permutation is $\sigma$ with
$\sigma(1)=2$, $\sigma(2)=3$ and $\sigma(3)=1$.
Therefore, its $6$-sign is $-1$.

We define the
operators $L_1,L_2,L_3,M_1,M_2,R$
on some flows $\Gamma$ of\linebreak
$\F_i(a_1,\ldots,a_q;b_1,\ldots,b_q)$ as follows.

\bigskip

\begin{tabular}{|r|l|l|}
\hline
\multicolumn{1}{|c|}{Graph} & \multicolumn{1}{c|}{Condition} & \multicolumn{1}{c|}{Operation}\\
\hhline{|=|=|=|}
$L_1(\Gamma)$ & $q\ge1$,                                 & replacement of $(s,a_q{-}1)$ \\
              & $a_q{-}1$ is a transit point of $\Gamma$ & with $(s,a_q)$               \\
\hline
$L_2(\Gamma)$ & $q\ge1$,                                 & addition of $(a_q{-}1,a_q)$\\
              & no edge of $\Gamma$ begins at $a_q{-}1$  &                                \\
\hline
$L_3(\Gamma)$ & $q\ge1$,                                 & replacement of $(a_q,t)$   \\
              & $a_q{-}1$ is no source of $\Gamma$       & with $(a_q-1,t)$           \\
\hhline{|=|=|=|}
$M_1(\Gamma)$ & $i+1$ is no sink of $\Gamma$,                      & replacement of $(s,i)$ \\
              & $i$ is a transit point of $\Gamma$       & with $(s,i+1)$         \\
\hline
$M_2(\Gamma)$ & $i+1$ is no sink of $\Gamma$,                                & addition of $(i,i+1)$  \\
              & no edge of $\Gamma$ begins at $i$                &                        \\
\hhline{|=|=|=|}
$R(\Gamma)$  & $q\ge1$                                           & replacement of $(s,b_q)$ \\
             & $b_q+1$ is no sink of $\Gamma$                    & with $(s,b_q+1)$         \\
\hline
\multicolumn{3}{c}{\footnotesize Table 1}
\end{tabular}

\bigskip

In each row of this table, the graph in the left column is obtained from $\Gamma$ by the operation
described in the right column (here $s$ and $t$ can be chosen uniquely) if $\Gamma$ satisfies the condition in the middle column.

{\it Example.} Let $\Gamma$ be as in the previous example.
Then $L_1(\Gamma)$ has edges $(5,9)$, $(3,8)$, $(6,7)$, $(4,6)$, $(2,5)$, $(1,3)$;
$L_2(\Gamma)$ is not well-defined, since $a_q-1=a_3-1=4$ is a transit point of $\Gamma$;
$L_3(\Gamma)$ has edges $(4,9)$, $(3,8)$, $(6,7)$, $(4,6)$, $(2,4)$, $(1,3)$ and is no flow
(see Lemma~\ref{lemma:L3}).


\begin{lemma}\label{lemma:L1}
If well-defined, $L_1(\Gamma)$ belongs to
$\F_i(a_1,{\ldots},a_{q-1},a_q-1;$\linebreak
$b_1,{\ldots},b_q)$
and has $i$-sign opposite to that of $\Gamma$.
\end{lemma}
\begin{proof} In $L_1(\Gamma)$, $a_q$ becomes a transit point
and $a_q-1$ becomes a source. Therefore, $L_1(\Gamma)$ and $\Gamma$
have the same sinks and number of transit points.


Let $\sigma$ be the linking permutation of $\Gamma$ and
$a_l$ be the source belonging to the connected component
of $\Gamma$ containing $a_q-1$. Clearly, $l<q$.

We set $a'_j:=a_j$ if $j=1,\ldots,q-1$ and $a'_q:=a_q-1$.
Then $a'_1,\ldots,a'_q$ are the sources of $L_1(\Gamma)$
arranged in the strictly ascending order.
We set $\tau:=\sigma\circ(l,q)$. One can easily see that
$a'_j$ and $b_{{\tau(j)}}$ are linked in $L_1(\Gamma)$
for any $j=1,\ldots,q$, i.e. $\tau$ is
the linking permutation of $L_1(\Gamma)$.
\end{proof}

\begin{lemma}\label{lemma:L2}
If well-defined, $L_2(\Gamma)$ belongs to
$\F_i(a_1,{\ldots},a_{q-1},a_q-1;$\linebreak
$b_1,{\ldots},b_q)$
and has $i$-sign opposite to that of $\Gamma$.
\end{lemma}
\begin{proof} In $L_2(\Gamma)$, $a_q$ becomes a transit point
and $a_q-1$ becomes a source. Therefore, $L_2(\Gamma)$ has
the same sinks as $\Gamma$ and one more transit point than $\Gamma$.
Finally notice that $L_2(\Gamma)$ and $\Gamma$ have the same
linking permutations.
\end{proof}

\begin{lemma}\label{lemma:L3}
Suppose that $L_3(\Gamma)$ is
well-defined. Then $L_3(\Gamma)$ is a flow if and only if
$a_q-1$ is no transit point of $\Gamma$.
If $L_3(\Gamma)$ is a flow, then it 
belongs to 
$\F_i(a_1,{\ldots},a_{q-1};a_q{-}1,b_1,{\ldots},b_q)$
and has the same $i$-sign as $\Gamma$.
If $L_3(\Gamma)$ is no flow, then there is exactly one flow
$\Gamma'$ of $\F_i(a_1,\ldots,a_q;b_1,\ldots,b_q)$
distinct from $\Gamma$ such that $L_3(\Gamma)=L_3(\Gamma')$.
Moreover, $\sgn_i(\Gamma)=-\sgn_i(\Gamma')$.
\end{lemma}
\begin{proof} If $a_q-1$ is no transit point, then
$L_3(\Gamma)$ is a flow with the same linking permutation,
sinks and number of transit points. Hence
$\sgn_i(L_3(\Gamma))=\sgn_i(\Gamma)$.

Now consider the case where $a_q-1$ is a transit point of $\Gamma$.
Then there is an edge of $\Gamma$ of the form $(a_q-1,r)$.
On the other hand, there is also an edge of $\Gamma$ of the
form $(a_q,t)$, since $a_q$ is a source of $\Gamma$.
In $L_3(\Gamma)$, $(a_q,t)$ turns to $(a_q-1,t)$, whence
$L_3(\Gamma)$ has two edges $(a_q-1,t)$ and $(a_q-1,r)$
beginning at $a_q-1$ and $L_3(\Gamma)$ is no flow.

Consider the flow $\Gamma'$ that is obtained from $\Gamma$
by the removing the edges $(a_q-1,r)$ and $(a_q,t)$
and adding the edges $(a_q,r)$ and $(a_q-1,t)$.
It is straightforward to see that $\Gamma'$ is the required flow.
To calculate the $i$-sign of $\Gamma'$, denote by $l$
the integer such that $a_l$ belongs to the connected
component of $\Gamma$ containing $r$ and denote by
$\sigma$ the linking permutation of $\Gamma$. Clearly, $l<q$.
Similarly to Lemma~\ref{lemma:L2}, one can see that
$\sigma\circ(l,q)$ is the linking permutation of $\Gamma'$.
Moreover, $\Gamma$ and $\Gamma'$ have the same transit points
and sinks.
\end{proof}

\begin{lemma}\label{lemma:M1}
If well-defined, $M_1(\Gamma)$ belongs to
$\F_i(a_1,{\ldots},a_q,i;
b_1,{\ldots},b_q,i{+}1)$
and has $i$-sign opposite to that of $\Gamma$.
\end{lemma}
\begin{proof} For $M_1(\Gamma)$, $i$ is no longer a transit point
(but a source) and $i+1$ is a new sink. Therefore the second factor
of the right-hand side of~(\ref{equation:lcwman:0}) does not change.

Let $\sigma$ be the linking permutation of $\Gamma$ and
$a_l$ be the source belonging to the connected component
of $\Gamma$ containing $i$. Note that $a_1,{\ldots},a_q,i$
are all the sources of $M_1(\Gamma)$ written in the strictly
ascending order and $b_1,{\ldots},b_q,i+1$ are all the sinks
of $M_1(\Gamma)$ written in the strictly descending order.
Therefore, under the natural identification $S_q<S_{q+1}$, 
we obtain that $\sigma\circ(l,q+1)$ is the linking permutation 
of $M_1(\Gamma)$.
\end{proof}

\begin{lemma}\label{lemma:M2}
If well-defined, $M_2(\Gamma)$ belongs to
$\F_i(a_1,{\ldots},a_q,i;
b_1,{\ldots},b_q,i{+}1)$
and has $i$-sign opposite to that of $\Gamma$.
\end{lemma}
\begin{proof} Note that $a_1,{\ldots},a_q,i$
are all the sources of $M_2(\Gamma)$ written in the strictly
ascending order and $b_1,{\ldots},b_q,i+1$ are all the sinks
of $M_2(\Gamma)$ written in the strictly descending order.
Hence the linking permutation of $M_2(\Gamma)$ is the same
as that of $\Gamma$ under the natural identification $S_q<S_{q+1}$.
Moreover, $M_2(\Gamma)$ has the same transit points as $\Gamma$.
\end{proof}

\begin{lemma}\label{lemma:R}
If well-defined, $R(\Gamma)$ belongs to
$\F_i(a_1,{\ldots},a_q;
b_1,{\ldots},b_{q-1},b_q{+}1)$
and has $i$-sign opposite to that of $\Gamma$.
\end{lemma}
\begin{proof} Note that $R(\Gamma)$ has
the same transit points and sources as $\Gamma$.
Moreover, $b_1,{\ldots},b_{q-1},b_q+1$ are all the sinks
of $M_2(\Gamma)$ written in the strictly descending order.
In particular, $R(\Gamma)$ and $\Gamma$ have the same
linking permutation.
\end{proof}

Now we are going to define ``inverse maps'' for each group of operators
$L_1$, $L_2$, $L_3$; $M_1$, $M_2$; $R$.

Suppose that $q{\ge}1$, $a_q{-}1\notin\{a_1,\ldots,a_{q-1}\}$ and
$\Gamma'\in\F_i(a_1,\ldots,a_{q-1},a_q{-}1;$\linebreak
$b_1,\ldots,b_q)$.
We define the operator $\L^o(\Gamma')$ that is one of the operators $L_1$, $L_2$, $L_3$ described in Table~1
and the graph $\L^g(\Gamma')$ using the table below.

\bigskip

\begin{tabular}{|l|c|l|}
\hline
\multicolumn{1}{|c|}{Condition} & \multicolumn{1}{c|}{}  & \multicolumn{1}{c|}{Value}\\
\hhline{|=|=|=|}
\multicolumn{1}{|l|}{there is an edge $(s,a_q)$ of $\Gamma'$}   &\multicolumn{1}{c|}{$\L^o(\Gamma')$}&\multicolumn{1}{c|}{$L_1$}\\
\cline{2-3}
 with $s<a_q-1$                            &$\L^g(\Gamma')$& obtained from $\Gamma'$ by replacing \\
                                           &               & $(s,a_q)$ with $(s,a_q-1)$\\
\hhline{|=|=|=|}
\multicolumn{1}{|l|}{$(a_q-1,a_q)$ is an edge of $\Gamma'$} &\multicolumn{1}{c|}{$\L^o(\Gamma')$}&\multicolumn{1}{c|}{$L_2$} \\
\cline{2-3}
                                                            &        $\L^g(\Gamma')$             &  obtained from $\Gamma'$ by\\
                                                            &                                    &  removing $(a_q-1,a_q)$\\
\hhline{|=|=|=|}
\multicolumn{1}{|l|}{$a_q$ is no transit point of $\Gamma'$} &\multicolumn{1}{c|}{$\L^o(\Gamma')$}&\multicolumn{1}{c|}{$L_3$} \\
\cline{2-3}
                                                            &        $\L^g(\Gamma')$             &  obtained from $\Gamma'$ by replacing \\
                                                            &                                    &  $(a_q-1,t)$ with $(a_q,t)$\\
\hline
\multicolumn{3}{c}{\footnotesize Table 2}
\end{tabular}

\bigskip
Let us take any row of this table such that the condition in its left column is satisfied.
If the middle column of this row contains $\L^g(\Gamma')$, then
this graph
is obtained from the flow $\Gamma'$ by the operation described in the right column. If the middle column
contains $\L^o(\Gamma')$, then
this operator
equals the operator in the right column.

{\it Example.} Let $q=2$, $i=3$, $a_1=1$, $a_2=3$, $b_2=4$, $b_1=5$ and $\Gamma'$ has edges
$(2,5)$, $(3,4)$, $(1,3)$. Then $\L^o(\Gamma')=L_1$ and $\L^g(\Gamma')$ has edges
$(2,5)$, $(3,4)$, $(1,2)$ (two top rows of Table~2 were applied). If we apply $\L^o(\Gamma')$
to $\L^g(\Gamma')$, then we obtain $\Gamma'$ (see the definition of $L_1$ in Table~1).

\begin{lemma}\label{lemma:lcwman:1}
Suppose that $q{\ge}1$ and $a_q{-}1\notin\{a_1,{\ldots},a_{q-1}\}$.
For any flow $\Gamma'\in 
\F_i(a_1,{\ldots},a_{q-1},a_q{-}1;
b_1,\ldots,b_q)$, we have
$\L^g(\Gamma')\in\F_i(a_1,{\ldots},a_q;b_1,{\ldots},b_q)$ and
$\L^o(\Gamma')\bigl(\L^g(\Gamma')\bigr)=\Gamma'$.
For any flow $\Gamma\in\F_i(a_1,\ldots,a_q;b_1,\ldots,b_q)$ and
integer $k=1,2,3$, we have $\L^o(L_k(\Gamma))=L_k$ and
$\L^g(L_k(\Gamma))=\Gamma$ if $L_k(\Gamma)$ is well-defined and a flow.
\end{lemma}
\begin{proof}
The result follows directly from Tables~1 and 2.
\end{proof}

\begin{corollary}\label{corollary:lcwman:1}
Suppose that $q\ge1$ and $a_q-1\notin\{a_1,\ldots,a_{q-1}\}$.
For any flow $\Gamma'\in\F_i(a_1,{\ldots},a_{q-1},a_q-1;
b_1,\ldots,b_q)$, there exists a unique pair
$(O,\Gamma)\in 
\{L_1,L_2,L_3\}\times\F_i(a_1,\ldots,a_q;b_1,\ldots,b_q)$
such that $O(\Gamma)=\Gamma'$.
\end{corollary}
\begin{proof}
It follows from the first assertion of Lemma~\ref{lemma:lcwman:1}
that $(\L^o(\Gamma'),\L^g(\Gamma'))$ is a suitable pair.
Now suppose that $(L_{k_1},\Gamma_1)$ and $(L_{k_2},\Gamma_2)$ are
pairs of $\{L_1,L_2,L_3\}\times\F_i(a_1,\ldots,a_q;b_1,\ldots,b_q)$
such that $L_{k_1}(\Gamma_1)=L_{k_2}(\Gamma_2)=\Gamma'$.
By the second assertion of Lemma~\ref{lemma:lcwman:1}, we have
$$
\begin{array}{lclclcl}
L_{k_1} &=&\L^o(L_{k_1}(\Gamma_1))&=&\L^o(L_{k_2}(\Gamma_2))&=&L_{k_2},\\[6pt]
\Gamma_1&=&\L^g(L_{k_1}(\Gamma_1))&=&\L^g(L_{k_2}(\Gamma_2))&=&\Gamma_2.
\end{array}
$$
\end{proof}

Now suppose that $a_q<i$, $i+1<b_q$ if $q\ge1$ and
$\Gamma'\in\F_i(a_1,\ldots,a_q,i;$\linebreak
$b_1,\ldots,b_q,i+1)$.
Then we use the table below to define the operator $\M^o(\Gamma')$ and the graph $\M^g(\Gamma')$.

\bigskip

\begin{tabular}{|l|c|l|}
\hline
\multicolumn{1}{|c|}{Condition} & \multicolumn{1}{c|}{}  & \multicolumn{1}{c|}{Value}\\
\hhline{|=|=|=|}
\multicolumn{1}{|l|}{there is an edge $(s,i+1)$}   &\multicolumn{1}{c|}{$\M^o(\Gamma')$}&\multicolumn{1}{c|}{$M_1$}\\
\cline{2-3}
 of $\Gamma'$ with $s<i$                            &$\M^g(\Gamma')$& obtained from $\Gamma'$ by replacing \\
                                           &               & $(s,i+1)$ with $(s,i)$\\
\hhline{|=|=|=|}
\multicolumn{1}{|l|}{$(i,i+1)$ is an edge of $\Gamma'$} &\multicolumn{1}{c|}{$\M^o(\Gamma')$}&\multicolumn{1}{c|}{$M_2$} \\
\cline{2-3}
                                                            &        $\M^g(\Gamma')$             &  obtained from $\Gamma'$ by\\
                                                            &                                    &  removing $(i,i+1)$\\
\hline
\multicolumn{3}{c}{\footnotesize Table 3}
\end{tabular}

\bigskip

This table defines the new graphs and the new
operators
in the same way as Table~2.

\begin{lemma}\label{lemma:lcwman:2}
Suppose that $a_q<i$, $i+1<b_q$ if $q\ge1$.
For any flow $\Gamma'\in$\linebreak
$\F_i(a_1,{\ldots},a_q,i;
b_1,\ldots,b_q,i+1)$, we have
$\M^g(\Gamma'){\in}\F_i(a_1,{\ldots},a_q;b_1,{\ldots},b_q)$ and
$\M^o(\Gamma')\bigl(\M^g(\Gamma')\bigr)=\Gamma'$.
For any flow $\Gamma\in\F_i(a_1,\ldots,a_q;b_1,\ldots,b_q)$ and
integer $k=1,2$, we have $\M^o(M_k(\Gamma))=M_k$ and
$\M^g(M_k(\Gamma))=\Gamma$ if $M_k(\Gamma)$ is well-defined.
\end{lemma}
\begin{proof}
The result follows directly from Tables~1 and 3.
\end{proof}

\begin{corollary}\label{corollary:lcwman:2}
Suppose that $a_q<i$, $i+1<b_q$ if $q\ge1$.
For any flow $\Gamma'\in$\linebreak
$\F_i(a_1,{\ldots},a_q,i;
b_1,\ldots,b_q,i+1)$, there exists a unique pair
$(O,\Gamma)\in$ \linebreak
$\{M_1,M_2\}\times\F_i(a_1,\ldots,a_q;b_1,\ldots,b_q)$
such that $O(\Gamma)=\Gamma'$.
\end{corollary}
\begin{proof} The result follows from Lemma~\ref{lemma:lcwman:2}
similarly to Corollary~\ref{corollary:lcwman:1}.
\end{proof}

Finally suppose that $q\ge1$, $b_q+1\notin\{b_{q-1},\ldots,b_1\}$ and
$\Gamma'\in\F_i(a_1,\ldots,a_q;$\linebreak
$b_1,\ldots,b_{q-1},b_q+1)$. Then we denote by $\R^g(\Gamma')$
the flow obtained from $\Gamma'$ by replacing the edge $(s,b_q+1)$
with $(s,b_q)$.

\begin{lemma}\label{lemma:lcwman:3}
Suppose that $q{\ge}1$ and $b_q{+}1\notin\{b_{q-1},{\ldots},b_1\}$.
For any flow $\Gamma'\in\F_i(a_1,\ldots,a_q;b_1,\ldots,b_{q-1},b_q+1)$, we have
$\R^g(\Gamma')\in\F_i(a_1,\ldots,a_q;b_1,\ldots,b_q)$ and
$R\bigl(\R^g(\Gamma')\bigr)=\Gamma'$.
For any flow $\Gamma\in\F_i(a_1,\ldots,a_q;b_1,\ldots,b_q)$,
we have
$\R^g(R(\Gamma))=\Gamma$.
\end{lemma}

\begin{corollary}\label{corollary:lcwman:3}
Suppose that $q{\ge}1$ and $b_q{+}1\notin\{b_{q-1},\ldots,b_1\}$.
For any flow $\Gamma'\in\F_i(a_1,\ldots,a_q;b_1,\ldots,b_{q-1},b_q+1)$,
there exists a unique flow
$\Gamma\in\F_i(a_1,\ldots,a_q;b_1,\ldots,b_q)$
such that $R(\Gamma)=\Gamma'$.
\end{corollary}
\begin{proof} The result follows from Lemma~\ref{lemma:lcwman:3}
similarly to Corollary~\ref{corollary:lcwman:1}.
\end{proof}


%
%


\subsection{Basis of the hyperalgebra}

In the remainder of the paper, we consider the case $G=A_\ell(\K)$.
We set $n:=\ell+1$, order the simple roots $\alpha_1,\ldots,\alpha_{n-1}$
so that they form the diagram

\begin{center}
\setlength{\unitlength}{1.2mm}
\begin{picture}(45,0)
\put(0,0){\circle{1}}
\put(10,0){\circle{1}}
\put(20,0){\circle{1}}
\put(40,0){\circle{1}}
\put(50,0){\circle{1}}
\put(0.5,0){\line(1,0){9}}
\put(10.5,0){\line(1,0){9}}
\put(40.5,0){\line(1,0){9}}
\put(20.5,0){\line(1,0){4}}
\put(39.5,0){\line(-1,0){4}}

\put(27.25,0){\circle{0}}
\put(30,0){\circle{0}}
\put(32.75,0){\circle{0}}

\put(-1,-3.5){$\scriptstyle\alpha_1$}
\put(9,-3.5){$\scriptstyle\alpha_2$}
\put(19,-3.5){$\scriptstyle\alpha_3$}
\put(49,-3.5){$\scriptstyle\alpha_{n-1}$}
\put(39,-3.5){$\scriptstyle\alpha_{n-2}$}
\end{picture}
\end{center}

\hspace{20mm}

\noindent
and use the 
short hand notation
$$
\llbracket i_1,\ldots,i_k\rrbracket_i:=[\eta_{\alpha_{i_1}},\ldots,\eta_{\alpha_{i_k}},\theta_{\omega_i}],\label{ll}
$$
where $i_1,\ldots,i_k,i$ belong to $\{1,\ldots,n-1\}$.
Here and in what follows, $\eta_\alpha:=\eta_{\alpha,1}$,
$[x,y]=xy-yx$ and long commutators are right-normed,
i.\;e. are defined by the inductive rule $[x_1,\ldots,x_k]=[x_1,[x_2,\ldots,x_k]]$ for $k>2$.



We denote by $\lm_i$ the $i$th entry of a sequence $\lm$ and
set $\delta_{\mathcal P}:=1$ if $\mathcal P$ is true
and $\delta_{\mathcal P}:=0$ if $\mathcal P$ is false
for an arbitrary condition $\mathcal P$.

In the present case, we can take $\mathfrak L=\mathfrak sl_n(\C)$
with the standard choice of the Chevalley basis:
$$
X_{\alpha_i+\cdots+\alpha_{j-1}}=E_{i,j},\quad\; X_{-\alpha_i-\cdots-\alpha_{j-1}}=E_{j,i},\quad\; H_{\alpha_i}=E_{i,i}-E_{i+1,i+1},
$$
where $1\le i<j\le n$. Here and in what follows, $E_{i,j}$ denotes the $n\times n$
matrix with $1$ in the intersection of row $i$ and column $j$ and $0$ elsewhere.

We denote $E_{i,j}^{(m)}:=(E_{i,j}^m/m!)\otimes 1_\K$ for $m\in\Z^+$ and
$H_i:=H_{\alpha_i}\otimes1_\K$.
This is simply more convenient
notation for the elements introduced in Section~\ref{Introduction}, in the case
$G=A_\ell(\K)$.
Indeed, for $1\le i<j\le n$ we have
$$
E_{i,j}^{(m)}=X_{\alpha_i+\cdots+\alpha_{j-1},m},\quad E_{j,i}^{(m)}=X_{-\alpha_i-\cdots-\alpha_{j-1},m},\quad H_i=H_{\alpha_i,1}.
$$
We also denote
$E_{i,j}^{(1)}$ by $E_{i,j}$. This notation does not cause confusion, since
we shall work only in $\U$.

We extend $E_{i,j}^{(m)}$ to negative values of $m$ by setting $E_{i,j}^{(m)}:=0$ if $m<0$.
This definition, useless at first sight, is in fact very significant, in particular,
because the formulas

\begin{equation}\label{equation:lcwman:0.5}
\begin{array}{l}
[E_{i,j},E_{j,k}^{(m)}]=
 E_{i,k}E_{j,k}^{(m-1)},\quad [E_{i,j},E_{k,i}^{(m)}]=
 -E_{k,j}E_{k,i}^{(m-1)},\\[6pt]
[E_{i,i+1},E_{i+1,i}^{(m)}]=E_{i+1,i}^{(m-1)}(H_i+1-m),\\[6pt]
H_lE_{i,j}^{(m)}=E_{i,j}^{(m)}(H_l+m(\delta_{l=i}-\delta_{l=j}-\delta_{l+1=i}+\delta_{l+1=j})),\\[6pt]
 E_{i,j}^{(m)}E_{i,j}=(m+1)E_{i,j}^{(m+1)}
\end{array}
\end{equation}
hold for any $m\in\Z$ and mutually distinct $i,j,k=1,\ldots,n$.
These formulas are trivial for $m<0$ and can be proved by induction on $m$ for $m\ge0$.

More generally, the first three formulas of~(\ref{equation:lcwman:0.5}) are a special case of formulas (2.1) from~\cite{Shchigolev},
which also hold for any integer values of the superscripts.
The reader is referred to~\cite{Graham_Knuth_Patashnik} to see why, for example,
setting ${\tbinom x m}=0$ for $m<0$ helps simplify a lot of calculations. The same thing
happens here: the proof of Lemma~\ref{lemma:lcwman:3.5} would be much longer if
we considered $E_{i,j}^{(m)}$ only for $m\ge0$.

\subsection{Action of $\eta_{\alpha_l}$}
Let $\UT(n)$ denote the set of
$n\times n$ matrices $N$ with entries in $\Z$ such that
$N_{a,b}=0$ unless $a<b$.
We denote by $N^s$ the sum of elements in row $s$ of $N$,
i.e. $N^s:=\sum_{b=1}^nN_{s,b}$.
For any $N\in\UT(n)$, we set
$$
F^{(N)}:=\prod\nolimits_{1\le a<b\le n}E_{b,a}^{(N_{a,b})},\label{FN}
$$
where
$E_{b,a}^{(N_{a,b})}$
precedes
$E_{d,c}^{(N_{c,d})}$
if and only if $b<d$ or $b=d$ and $a<c$.
%
It follows from Proposition~\ref{proposition:lcwman2:0} that
any element of $\U^{-,0}$ is 
representable as $\sum_{N\in\UT(n)}F^{(N)}\H_N$, where $\H_N\in\U^0$.

\begin{lemma}\label{lemma:lcwman:3.5}
Let $l=1,\ldots,n-1$ and $\H_N$, where $N\in\UT(n)$, be elements of $\U^0$
that are nonzero only for finitely many matrices $N$.
We have
\begin{equation}\label{equation:lcwman:1}
{\arraycolsep=0pt
\begin{array}{l}
\displaystyle
\eta_{\alpha_l}
\(
\sum_{N\in\UT(n)}F^{(N)}\H_N
\){=}
\sum_{N\in\UT(n)}F^{(N)}\Biggl(\sum_{1\le s <l}N_{s,l}\H_{N+E_{s,l+1}-E_{s,l}}\\[30pt]
\displaystyle+(H_l-N^l+N^{l+1})\H_{N+E_{l,l+1}}
-\sum_{l+1<t\le n}N_{l+1,t}\H_{N+E_{l,t}-E_{l+1,t}}\Biggr).
\end{array}}
\end{equation}
\end{lemma}
\begin{proof}
First we prove that
\begin{equation}\label{equation:lcwman:1.5}
\begin{array}{l}
\displaystyle [E_{l,l+1},F^{(N)}]=\sum_{1\le s<l}(N_{s,l}+1)F^{(N-E_{s,l+1}+E_{s,l})}+\\[12pt]
\displaystyle F^{(N-E_{l,l+1})}\bigl(H_l{+}1{-}N^l{+}N^{l+1}\bigr)
              -\sum_{l+1<t\le n}(N_{l+1,t}{+}1)F^{(N-E_{l,t}+E_{l+1,t})}
\end{array}
\end{equation}
for any $N\in\UT(n)$.
Consider the representation $F^{(N)}=F_2\cdots F_n$,
where
$F_j=E_{j,1}^{(N_{1,j})}\cdots E_{j,j-2}^{(N_{j-2,j})}E_{j,j-1}^{(N_{j-1,j})}$.
Since $E_{l,l+1}$ commutes with 
$F_2,\ldots,F_l$, we have
$E_{l,l+1}F^{(N)}=F_2\cdots F_lE_{l,l+1}F_{l+1}\cdots F_n$.
Applying~(\ref{equation:lcwman:0.5}), we obtain
\begin{equation}\label{equation:lcwman:1.5625}
\begin{array}{rcl}
\displaystyle E_{l,l+1}F_{l+1}&=&\displaystyle\sum\nolimits_{1\le s<l}\left(E_{l,s}
\prod\nolimits_{1\le r\le l}E_{l+1,r}^{(N_{r,l+1}-\delta_{r=s})}\right)+\\[12pt]
\displaystyle                    &&\displaystyle\(\prod\nolimits_{1\le r<l}E_{l+1,r}^{(N_{r,l+1})}\)E_{l,l+1}E_{l,l+1}^{(N_{l,l+1})}.
\end{array}
\end{equation}
Applying~(\ref{equation:lcwman:0.5}) again, we obtain
$$
{
\arraycolsep=0pt
\begin{array}{l}
E_{l,l+1}E_{l+1,l}^{(N_{l,l+1})}F_{l+2}\cdots F_n=
E_{l+1,l}^{(N_{l,l+1})}E_{l,l+1}F_{l+2}\cdots F_n\\[6pt]
+E_{l+1,l}^{(N_{l,l+1}-1)}(H_l+1-N_{l,l+1})F_{l+2}\cdots F_n=E_{l+1,l}^{(N_{l,l+1})}E_{l,l+1}F_{l+2}\cdots F_n\\[6pt]
+E_{l+1,l}^{(N_{l,l+1}-1)}F_{l+2}(H_l+1-N_{l,l+1}-N_{l,l+2}+N_{l+1,l+2})F_{l+3}\cdots F_n=\cdots\\[6pt]
=E_{l+1,l}^{(N_{l,l+1})}E_{l,l+1}F_{l+2}\cdots F_n+
E_{l+1,l}^{(N_{l,l+1}-1)}F_{l+2}\cdots F_n(H_l{+}1{-}N^l{+}N^{l+1}).
\end{array}
}
$$
Thus, multiplying~(\ref{equation:lcwman:1.5625}) by $F_2\cdots F_l$ on the left and
by $F_{l+2}\cdots F_n$ on the right, we obtain
$$
{
\arraycolsep=0pt
\begin{array}{l}
\displaystyle
E_{l,l+1}F^{(N)}{=}\!
\sum\nolimits_{1\le s<l}\!\left(F_2{\cdots}F_lE_{l,s}
\prod\nolimits_{1\le r\le l}E_{l+1,r}^{(N_{r,l+1}-\delta_{r=s})}\right)\!\times\\[6pt]
\displaystyle
\times F_{l+2}\cdots F_n
+F_2\cdots F_l\(\prod\nolimits_{1\le r<l}E_{l+1,r}^{(N_{r,l+1})}\)
E_{l+1,l}^{(N_{l,l+1}-1)}\times\\[12pt]
\times F_{l+2}\cdots F_n(H_l{+}1{-}N^l{+}N^{l+1})
+F_2\cdots F_{l+1}E_{l,l+1}F_{l+2}\cdots F_n=\\[6pt]
\displaystyle\sum\nolimits_{1\le s<l}(N_{s,l}+1)F^{(N-E_{s,l+1}+E_{s,l})}+
F^{(N-E_{l,l+1})}(H_l{+}1{-}N^l{+}N^{l+1})+\\[10pt]
F_2\cdots F_{l+1}E_{l,l+1}F_{l+2}\cdots F_n.
\end{array}
}
$$
In the above rearrangement, we used that
$$
F_lE_{l,s}=(N_{s,l}+1)\prod\nolimits_{1\le r<l}E_{l,r}^{(N_{r,l}+\delta_{r=s})}\mbox{ for }1\le s<l.
$$
It remains to calculate $E_{l,l+1}F_{l+2}\cdots F_n$.
Let $t=l+2,\ldots,n$. All factors of $F_t$ commute with
$E_{l,l+1}$ except $E_{t,l}^{(N_{l,t})}$.
Applying~(\ref{equation:lcwman:0.5}), we obtain
$$
\begin{array}{l}
E_{l,l+1}E_{t,l}^{(N_{l,t})}E_{t,l+1}^{(N_{l+1,t})}=
            -E_{t,l}^{(N_{l,t}-1)}E_{t,l+1}E_{t,l+1}^{(N_{l+1,t})}
            +E_{t,l}^{(N_{l,t})}E_{t,l+1}^{(N_{l+1,t})}E_{l,l+1}\\[6pt]
            =-(N_{l+1,t}+1)E_{t,l}^{(N_{l,t}-1)}
            E_{t,l+1}^{(N_{l+1,t}+1)}
            +E_{t,l}^{(N_{l,t})}
            E_{t,l+1}^{(N_{l+1,t})}E_{l,l+1}.
\end{array}
$$
Hence $E_{l,l+1}F_t=-(N_{l+1,t}+1)\prod\nolimits_{1\le r<t}E_{t,r}^{(N_{r,t}-\delta_{r=l}+\delta_{r=l+1})}+F_tE_{l,l+1}$
and~(\ref{equation:lcwman:1.5}) follows.


Now~(\ref{equation:lcwman:1}) follows from~(\ref{equation:lcwman:1.5})
and the equivalence
$$
E_{l,l+1}\sum\nolimits_{N\in\UT(n)}F^{(N)}\H_N\=
\sum\nolimits_{N\in\UT(n)}[E_{l,l+1},F^{(N)}]\H_N\!\!\!\!\!\pmod{\U\cdot E_{l,l+1}}.
$$
\end{proof}


We shall identify any graph $\Gamma$ with vertices $\Z$ and
finitely many edges all having the form $(s,t)$, where $1\le s<t\le n$,
with
the matrix of $\UT(n)$
such that $N_{s,t}$
equals the number of edges begging at $s$ and ending at $t$.

{\it Example.} Let $n=4$ and $\Gamma$ have edges $(1,2)$, $(1,4)$, $(2,3)$, $(3,4)$.
Then $\Gamma$ is identified with the matrix
$$
\left(\!
\begin{array}{cccc}
0&1&0&1\\
0&0&1&0\\
0&0&0&1\\
0&0&0&0
\end{array}
\!
\right)
$$


\begin{definition}\label{definition:lcwman2:2}
Let $1\,{\le}\,a\,{\le}\,i\,{\le}\,b\,{<}\,n$.\!\! A sequence of integers
satisfies $\O(a,i,b)$ if
this sequence is obtained by inserting the sequences $a,a{+}1,\ldots,i{-}1$
and $b,b{-}1,\ldots,i{+}1$ into one another, preserving the order in
either of them.
\end{definition}

\begin{remark}\label{remark:lcwman2:1.5}\!\!
The only sequence satisfying $\O(i,i,i)$ is the empty sequence~$\emptyset$.
\end{remark}

\medskip

{\it Example.} The sequence $\underline 1,8,\underline 2,\underline 3,7,6,\underline 4$ satisfies $\O(1,5,8)$.
We underlined the elements of the sequence $a,a{+}1,\ldots,i{-}1$.






\begin{lemma}\label{lemma:lcwman:4}
Fix some integers $1\le a_1<a_2<\cdots<a_q\le i<b_q<\cdots<b_2<b_1\le n$.
Let
\begin{equation}\label{equation:lcwman:1.625}
i^{(q)}_1,{\ldots},i^{(q)}_{k_q};\quad
i^{(q-1)}_1,\ldots,i^{(q-1)}_{k_{q-1}};
\quad\ldots\quad;\quad
i^{(1)}_1,\ldots,i^{(1)}_{k_1}
\end{equation}
be sequences satisfying $\O(a_q,i,b_q-1)$, $\O(a_{q-1},i,b_{q-1}-1)$, \ldots, $\O(a_1,i,b_1-1)$ respectively.
For $\H_N$ as in Lemma~\ref{lemma:lcwman:3.5}, we have
$$
{\arraycolsep=0pt
\begin{array}{l}
\displaystyle
\left\llbracket
i^{(q)}_1,{\ldots},i^{(q)}_{k_q},i,
i^{(q-1)}_1,{\ldots},i^{(q-1)}_{k_{q-1}},i,\;\ldots,\;
 i^{(1)}_1,{\ldots},i^{(1)}_{k_1},i
 \right\rrbracket
 _i\!\!
\left(\sum_{N\in\UT(n)}\!F^{(N)}\H_N\right)\!{=}\\[6pt]
\displaystyle
\sum_{N\in\UT(n)}F^{(N)}
\sum_{\Gamma\in\F_i(a_1,\ldots,a_q;b_1,\ldots,b_q)}
\sgn_i(\Gamma)\theta_{\omega_i}\(\H_{N+\Gamma}\).
\end{array}}
$$
\end{lemma}
\begin{proof} We put for brevity
$$
\begin{array}{rcl}
x&:=&
\left\llbracket
%
%
i^{(q)}_1,{\ldots},i^{(q)}_{k_q},i,
i^{(q-1)}_1,{\ldots},i^{(q-1)}_{k_{q-1}},i,\;\ldots,\;
i^{(1)}_1,{\ldots},i^{(1)}_{k_1},i
\right\rrbracket_i,\\[6pt]
\F&:=&\F_i(a_1,\ldots,a_q;b_1,\ldots,b_q).
\end{array}
$$


We apply induction on the length of the sequence
inside the brackets $\llbracket\;\rrbracket$ in the formulation of the current lemma
(the above formula).
%
If this length is zero (i.e. $q=0$), then the required equality holds since
$\llbracket\;\rrbracket_i=\theta_{\omega_i}$, $\F_i()$ consists of the empty flow
$\Gamma_\emptyset$, which is identified with the zero matrix, and
$\sgn_i(\Gamma_\emptyset)=1$.

Now suppose that the equality in the formulation of the current lemma holds.
Take any $l=1,\ldots,n-1$ and set
\begin{equation}\label{equation:lcwman:1.875}
y\,{:=}\,[\eta_{\alpha_l},x]\,{=}\!
\left\llbracket l,
i^{(q)}_1,{\ldots},i^{(q)}_{k_q},i,
i^{(q-1)}_1,{\ldots},i^{(q-1)}_{k_{q-1}},i,\;\ldots,\;
i^{(1)}_1,{\ldots},i^{(1)}_{k_1},i
%
\right\rrbracket_i\!.\!
%
\end{equation}
Applying~(\ref{equation:lcwman:1}) and the inductive hypothesis, we get
$$
\begin{array}{l}
\displaystyle
\eta_{\alpha_l}x\(\sum\nolimits_{N\in\UT(n)}F^{(N)}\H_N\)=\\[12pt]
\displaystyle
\eta_{\alpha_l}
\(
\sum\nolimits_{N\in\UT(n)}F^{(N)}\[
\sum\nolimits_{\Gamma\in\F}
\sgn_i(\Gamma)\theta_{\omega_i}\(\H_{N+\Gamma}\)\]
\)=\\[12pt]
\displaystyle
\sum\nolimits_{N\in\UT(n)}F^{(N)}
\biggl(
\sum\nolimits_{1\le s <l}N_{s,l}
\[
\sum\nolimits_{\Gamma\in\F}
\sgn_i(\Gamma)\theta_{\omega_i}(\H_{N+E_{s,l+1}-E_{s,l}+\Gamma})
\]\\[12pt]
\displaystyle
+(H_l-N^l+N^{l+1})
\[
\sum\nolimits_{\Gamma\in\F}
\sgn_i(\Gamma)\theta_{\omega_i}(\H_{N+E_{l,l+1}+\Gamma})
\]\\[6pt]
\displaystyle
-\sum\nolimits_{l+1<t\le n}N_{l+1,t}
\[
\sum\nolimits_{\Gamma\in\F}
\sgn_i(\Gamma)\theta_{\omega_i}(\H_{N+E_{l,t}-E_{l+1,t}+\Gamma})
\]
\biggr)=\\[12pt]
\displaystyle
\sum\nolimits_{N\in\UT(n)}F^{(N)}
\sum\nolimits_{\Gamma\in\F}
\biggl[
\sum\nolimits_{1\le s <l}\sgn_i(\Gamma)N_{s,l}
\theta_{\omega_i}(\H_{N+E_{s,l+1}-E_{s,l}+\Gamma})\\[18pt]
+\sgn_i(\Gamma)(H_l-N^l+N^{l+1})
 \theta_{\omega_i}(\H_{N+E_{l,l+1}+\Gamma})\\[12pt]
\displaystyle
-\sum\nolimits_{l+1<t\le n}\sgn_i(\Gamma)N_{l+1,t}
 \theta_{\omega_i}(\H_{N+E_{l,t}-E_{l+1,t}+\Gamma})
\biggr].
\end{array}
$$
On the other hand, we get
$$
\begin{array}{l}
\displaystyle
x\,\eta_{\alpha_l}\(\sum\nolimits_{N\in\UT(n)}F^{(N)}\H_N\)=\\[12pt]
\displaystyle
x
\biggl(
\sum\nolimits_{N\in\UT(n)}F^{(N)}\biggl[\sum\nolimits_{1\le s <l}N_{s,l}\H_{N+E_{s,l+1}-E_{s,l}}\\[18pt]
\displaystyle+(H_l-N^l+N^{l+1})\H_{N+E_{l,l+1}}
-\sum\nolimits_{l+1<t\le n}N_{l+1,t}\H_{N+E_{l,t}-E_{l+1,t}}\biggr]
\biggr)=\\[12pt]
\displaystyle
\sum\nolimits_{N\in\UT(n)}F^{(N)}\sum\nolimits_{\Gamma\in\F}
\biggl[\sum\nolimits_{1\le s <l}\sgn_i(\Gamma)\,(N{+}\Gamma)_{s,l}\,\theta_{\omega_i}(\H_{N+\Gamma+E_{s,l+1}-E_{s,l}})\\[18pt]
\displaystyle
+\sgn_i(\Gamma)\,(\theta_{\omega_i}(H_l)-(N+\Gamma)^l+(N+\Gamma)^{l+1})\,\theta_{\omega_i}(\H_{N+\Gamma+E_{l,l+1}})\\[12pt]
\displaystyle
-\sum\nolimits_{l+1<t\le n}\sgn_i(\Gamma)\,(N+\Gamma)_{l+1,t}\,\theta_{\omega_i}(\H_{N+\Gamma+E_{l,t}-E_{l+1,t}})\biggr].
\end{array}
$$
Subtracting the latter expression from the former, we obtain
(recall the definition of $\theta_\delta$ in Section~\ref{twisting operators})
$$
\begin{array}{l}
\displaystyle
y\(\sum\nolimits_{N\in\UT(n)}F^{(N)}\H_N\)=\\[12pt]
\displaystyle
\sum\nolimits_{N\in\UT(n)}F^{(N)}\sum\nolimits_{\Gamma\in\F}
\biggl[\sum\nolimits_{1\le s <l}-\sgn_i(\Gamma)\,\Gamma_{s,l}\,\theta_{\omega_i}(\H_{N+\Gamma+E_{s,l+1}-E_{s,l}})\\[18pt]
\displaystyle
+\sgn_i(\Gamma)\,(-\delta_{i=l}+\Gamma^l-\Gamma^{l+1})\,
 \theta_{\omega_i}(\H_{N+\Gamma+E_{l,l+1}})+\\[12pt]
\displaystyle
\sum\nolimits_{l+1<t\le n}\sgn_i(\Gamma)\,\Gamma_{l+1,t}\,
\theta_{\omega_i}(\H_{N+\Gamma+E_{l,t}-E_{l+1,t}})
\biggr].
\end{array}
$$
We denote by $S_1(\Gamma)$, $S_2(\Gamma)$ and $S_3(\Gamma)$ the first,
the second and the third summands respectively
in the square brackets in the above formula.

Now we want to choose $l$ so that the sequence in the brackets $\llbracket\;\rrbracket$
of~(\ref{equation:lcwman:1.875}) could appear in the brackets $\llbracket\;\rrbracket$
in the formulation of the current lemma for different
sequences~(\ref{equation:lcwman:1.625}) and possibly different $q$ (Case~2).

{\it Case 1: $q\ge1$ and $1\le l=a_q-1\notin\{a_1,\ldots,a_{q-1}\}$.}
The values of $S_1(\Gamma)$, $S_2(\Gamma)$ and $S_3(\Gamma)$
are given in the table below.

%
%
%

\bigskip
{
\extrarowheight=3pt
\begin{tabular}{|c|c|c|}
\hline
  &$a_q{-}1$ is no transit point of $\Gamma$ & $a_q{-}1$ is a transit point of $\Gamma$\\[2pt]
\hline
$S_1(\Gamma)$  & 0 & $-\sgn_i(\Gamma)\,\theta_{\omega_i}\bigl(\H_{N+L_1(\Gamma)}\bigr)$\\[3pt]
\hline
$S_2(\Gamma)$  &$-\sgn_i(\Gamma)\,\theta_{\omega_i}\bigl(\H_{N+L_2(\Gamma)}\bigr)$&0\\[3pt]
\hline
$S_3(\Gamma)$  & \;\;\;$\sgn_i(\Gamma)\,\theta_{\omega_i}\bigl(\H_{N+L_3(\Gamma)}\bigr)$ & \;\;\;$\sgn_i(\Gamma)\,\theta_{\omega_i}\bigl(\H_{N+L_3(\Gamma)}\bigr)$\\[3pt]
\hline
\multicolumn{3}{c}{\footnotesize Table 4}
\end{tabular}}

\medskip

\noindent
In each row of this table (except the top row), the element of $\U$ in the left column
equals either the element in the middle column or the element in the right column depending
on which condition in the top row of this table is satisfied for $\Gamma$.


To verify this table, one should note that in the present case
$l<i$, $\Gamma^{l+1}=1$ (since $l+1=a_q$ is a source of $\Gamma$) and
$\Gamma^l=0$ if $a_q-1$ is no transit point of $\Gamma$ and
$\Gamma^l=1$ otherwise.
By Lemmas~\ref{lemma:L1}--\ref{lemma:L3}, every nonzero cell of this table
(not in the top row and not in the left column)
that is not the bottom rightmost
yields a flow of $\F_i(a_1,\ldots,a_{q-1},a_q-1;b_1,\ldots,b_q)$ as the second summand in the subscript of $\H$.
Therefore, applying Lemmas~\ref{lemma:L1}--\ref{lemma:L3}
and Corollary~\ref{corollary:lcwman:1}, we obtain
$$
\begin{array}{l}
\displaystyle
y\(\sum\nolimits_{N\in\UT(n)}F^{(N)}\H_N\)=\\[12pt]
\displaystyle
\sum\nolimits_{N\in\UT(n)}F^{(N)}\biggl(\sum\Bigl\{\sgn_i\bigl(L_k(\Gamma)\bigr)\,\theta_{\omega_i}\bigl(\H_{N+L_k(\Gamma)}\bigr)\,\Bigl|\,\Gamma\in\F\; \& \;k=1,2,3\\[10pt]
\&\; L_k(\Gamma)\mbox{ is well-defined and a flow}\Bigr\}+\\[6pt]
\displaystyle\sum\Bigl\{\sgn_i(\Gamma)\,\theta_{\omega_i}\bigl(\H_{N+L_3(\Gamma)}\bigr)\,\Bigl|\,
\Gamma\in\F \; \& \; a_q{-}1\mbox{ is a transit point of }\Gamma
\Bigr\}\biggr)=\\[6pt]
\end{array}
$$
$$
\begin{array}{l}
\displaystyle
\sum\nolimits_{N\in\UT(n)}F^{(N)}\biggl(
\sum\nolimits_{\Gamma'\in\F_i(a_1,\ldots,a_{q-1},a_q-1;b_1,\ldots,b_q)}
\sgn_i(\Gamma')\,\theta_{\omega_i}(\H_{N+\Gamma'})+\\[12pt]
%
\displaystyle\sum\Bigl\{\sgn_i(\Gamma)\,\theta_{\omega_i}\bigl(\H_{N+L_3(\Gamma)}\bigr)\,\Bigl|\,
\Gamma\in\F \; \& \; a_q{-}1\mbox{ is a transit point of }\Gamma
\Bigr\}
\biggr).
\end{array}
$$
By Lemma~\ref{lemma:L3}, the last summand of the above formula equals zero.
Thus we have proved the current lemma for the sequences
\begin{equation}\label{equation:lemma:lcwman:4:case:1}
{a_q{-}1},i^{(q)}_1,{\ldots},i^{(q)}_{k_q};\quad
i^{(q-1)}_1,\ldots,i^{(q-1)}_{k_{q-1}};\quad\ldots\quad;\quad
i^{(1)}_1,\ldots,i^{(1)}_{k_1}
\end{equation}
satisfying conditions $\O(a_q-1,i,b_q-1)$, $\O(a_{q-1},i,b_{q-1}-1)$, \ldots, $\O(a_1,i,b_1-1)$ respectively.

{\it Case 2: $a_q<i$, $i+1<b_q$ if $q\ge1$ and $l=i$.}
The values of $S_1(\Gamma)$, $S_2(\Gamma)$ and $S_3(\Gamma)$
are given in the table below.

\bigskip
{
\extrarowheight=2pt
\begin{tabular}{|c|c|c|}
\hline
  &$i$ is no transit point of $\Gamma$ & $i$ is a transit point of $\Gamma$\\
\hline
$S_1(\Gamma)$  & $0$ & $-\sgn_i(\Gamma)\,\theta_{\omega_i}\bigl(\H_{N+M_1(\Gamma)}\bigr)$\\[3pt]
\hline
$S_2(\Gamma)$  &$-\sgn_i(\Gamma)\,\theta_{\omega_i}\bigl(\H_{N+M_2(\Gamma)}\bigr)$& $0$\\[3pt]
\hline
$S_3(\Gamma)$  & $0$ & $0$\\
\hline
\multicolumn{3}{c}{\footnotesize Table 5}
\end{tabular}}

\noindent
This table should be interpreted in the same way as Table~4.

To verify this table, one should note that in the present case
$\Gamma_{l+1,t}=0$ for any $t\in\Z$, $\Gamma^{l+1}=0$,
$\Gamma^l=0$ if $i$ is no transit point of $\Gamma$ and
$\Gamma^l=1$ otherwise. Therefore, applying
Lemmas~\ref{lemma:M1} and \ref{lemma:M2}
and Corollary~\ref{corollary:lcwman:2}, we obtain
$$
\begin{array}{l}
\displaystyle
y\(\sum\nolimits_{N\in\UT(n)}F^{(N)}\H_N\)=\\[12pt]
\displaystyle
\sum\nolimits_{N\in\UT(n)}F^{(N)}\sum\Bigl\{\sgn_i\bigl(M_k(\Gamma)\bigr)\,\theta_{\omega_i}\bigl(\H_{N+M_k(\Gamma)}\bigr)\,\Bigl|\,\Gamma\in\F\; \& \;k=1,2\\[10pt]
\&\; M_k(\Gamma)\mbox{ is well-defined}\Bigr\}=\\[6pt]
\displaystyle
\sum\nolimits_{N\in\UT(n)}F^{(N)}
\sum\nolimits_{\Gamma'\in\F_i(a_1,\ldots,a_q,i;b_1,\ldots,b_q,i+1)}
\sgn_i(\Gamma')\,\theta_{\omega_i}(\H_{N+\Gamma'}).
\end{array}
$$
Thus we have proved the current lemma for the $q+1$ sequences
\begin{equation}\label{equation:lemma:lcwman:4:case:2}
\emptyset;\quad i^{(q)}_1,{\ldots},i^{(q)}_{k_q};\quad
i^{(q-1)}_1,\ldots,i^{(q-1)}_{k_{q-1}};\quad\ldots\quad;\quad
i^{(1)}_1,\ldots,i^{(1)}_{k_1}
\end{equation}
satisfying conditions $\O(i,i,i)$, $\O(a_q,i,b_q-1)$, $\O(a_{q-1},i,b_{q-1}-1)$, \ldots, \linebreak$\O(a_1,i,b_1-1)$ respectively
(see Remark~\ref{remark:lcwman2:1.5}).

{\it Case 3: $q\ge1$, $b_q+1\notin\{b_{q-1},\ldots,b_1\}$ and $l=b_q$.}
We have $i<l$, $\Gamma_{l+1,t}=0$ for any $t\in\Z$,
$\Gamma^{l+1}=0$, $\Gamma^l=0$. Hence
$S_1(\Gamma)=-\sgn_i(\Gamma)\theta_{\omega_i}(\H_{N+R(\Gamma)})$,
$S_2(\Gamma)=0$ and $S_3(\Gamma)=0$.
Therefore, applying Lemma~\ref{lemma:R} and
Corollary~\ref{corollary:lcwman:3}, we obtain
$$
\begin{array}{l}
\displaystyle
y\(\sum\nolimits_{N\in\UT(n)}F^{(N)}\H_N\)=\\[12pt]
\displaystyle\sum\nolimits_{N\in\UT(n)}F^{(N)}\sum\nolimits_{\Gamma\in\F}\sgn_i\bigl(R(\Gamma)\bigr)\theta_{\omega_i}\bigl(\H_{N+R(\Gamma)}\bigr)=\\[10pt]
\displaystyle
\sum\nolimits_{N\in\UT(n)}F^{(N)}
\sum\nolimits_{\Gamma'\in\F_i(a_1,\ldots,a_q;b_1,\ldots,b_{q-1},b_q+1)}
\sgn_i(\Gamma')\theta_{\omega_i}(\H_{N+\Gamma'}).
\end{array}
$$

Thus we have proved the current lemma for the sequences
\begin{equation}\label{equation:lemma:lcwman:4:case:3}
b_q,i^{(q)}_1,{\ldots},i^{(q)}_{k_q};\quad
i^{(q-1)}_1,\ldots,i^{(q-1)}_{k_{q-1}};\quad\ldots\quad;\quad
i^{(1)}_1,\ldots,i^{(1)}_{k_1}
\end{equation}
satisfying conditions $\O(a_q,i,b_q)$,\, $\O(a_{q-1},i,b_{q-1}-1)$,\; \ldots,\; $\O(a_1,i,b_1-1)$ \linebreak respectively.

To finish the proof, it suffices to notice that any sequence of sequences in the hypothesis
of the current lemma can be obtained by passing from~(\ref{equation:lcwman:1.625})
to~(\ref{equation:lemma:lcwman:4:case:1}) or~(\ref{equation:lemma:lcwman:4:case:2})
or~(\ref{equation:lemma:lcwman:4:case:3}).
\end{proof}

\begin{definition}\label{definition:xi}\!
It follows from Lemma~\ref{lemma:lcwman:4} that the operator\!
$\left\llbracket i^{(q)}_1,{\ldots},i^{(q)}_{k_q},i,\right.$\!\linebreak
$\left.i^{(q-1)}_1,{\ldots},i^{(q-1)}_{k_{q-1}},i,{\ldots},
 i^{(1)}_1,{\ldots},i^{(1)}_{k_1},i\right\rrbracket_i$ depends only
on the integers $a_1,{\ldots},a_q$,\linebreak
$b_1,\ldots,b_q$ and $i$.
We, therefore, denote it by $\xi_i(a_1,\ldots,a_q;b_1,\ldots,b_q)$.
\end{definition}

%
%

\subsection{Tableaux}\label{Tableaux}
A {\it composition} of length $n$
is a sequence $\lm=(\lm_1,\ldots,\lm_n)$ such that $\lm_1,\ldots,\lm_n\in\Z^+$.
If, moreover, $\lm_1\ge\cdots\ge\lm_n$ then $\lm$ is called a {\it partition}.
The {\it diagram} of a composition $\lm$ is the set
$$
[\lm]=\{(i,j)\in\Z^2\|1\le i\le n\mbox{ and }1\le j\le \lm_i\}.\label{diagram}
$$
We shall think of $[\lm]$ as an array of boxes.
For example, if $\lm{=}(5,3,2,0)$ then

{
\begin{center}
\setlength{\unitlength}{0.4mm}
\begin{picture}(45,0)
\put(0,0){\line(1,0){50}}
\put(0,-10){\line(1,0){50}}
\put(0,-20){\line(1,0){30}}
\put(0,-30){\line(1,0){20}}
\put(0,0){\line(0,-1){30}}
\put(10,0){\line(0,-1){30}}
\put(20,0){\line(0,-1){30}}
\put(30,0){\line(0,-1){20}}
\put(40,0){\line(0,-1){10}}
\put(50,0){\line(0,-1){10}}
\put(-25,-17){$[\lambda]=$}
\end{picture}
\end{center}
}

\vspace{1.3cm}

A {\it $\lm$-tableau}, where $\lm$ is a composition, is a function $T:[\lm]\to\{1,\ldots,n\}$,
which we regard as the diagram $[\lm]$ filled with
integers in $\{1,\ldots,n\}$.
%
%
A $\lm$-tableau $T$ is called {\it row standard} if its entries weakly
increase along the rows, that is $T(i,j)\le T(i,j')$ if $j<j'$.
A $\lm$-tableau $T$ is called {\it regular row standard} if it is
row standard and every entry in row $i$ of $T$ is at least~$i$.
Finally, if $\lm$ is partition, then a $\lm$-tableau $T$
is called {\it standard} if it is row standard
and its entries strictly increase down the columns,
that is $T(i,j)<T(i',j)$ if $i<i'$. For example,

{
\begin{center}
\setlength{\unitlength}{0.4mm}
\begin{picture}(45,0)
\put(0,0){\line(1,0){50}}
\put(0,-10){\line(1,0){50}}
\put(0,-20){\line(1,0){30}}
\put(0,-30){\line(1,0){20}}
\put(0,0){\line(0,-1){30}}
\put(10,0){\line(0,-1){30}}
\put(20,0){\line(0,-1){30}}
\put(30,0){\line(0,-1){20}}
\put(40,0){\line(0,-1){10}}
\put(50,0){\line(0,-1){10}}
\put(-25,-17){$T=$}
\put(3,-8.5){1} \put(13,-8.5){2} \put(23,-8.5){2} \put(33,-8.5){3}\put(43,-8.5){4}
\put(3,-18.5){2}\put(13,-18.5){3}\put(23,-18.5){3}
\put(3,-28.5){3}\put(13,-28.5){4}
\end{picture}
\end{center}
}

\vspace{1.3cm}

\noindent
is a standard $(5,3,2,0)$-tableau. 
Here $n=4$.

For any regular row standard $\lm$-tableau $T$, we define $\N_T$\label{NT}
to be the matrix of $\UT(n)$ whose $(i,j)$-entry, where $i<j$, is the number
of entries $j$ in row $i$ of $T$ and set $F_T:=F^{(\N_T)}$.
For $n=4$ and $T$ as in the above example, we have
$$
\N_T{=}\!\left(
\begin{array}{llll}
0&2&1&1\\
0&0&2&0\\
0&0&0&1\\
0&0&0&0\\
\end{array}
\right)\!\;
\mbox{ and }\;
 F_T=E_{2,1}^{(2)}\,  E_{3,1}\,               E_{3,2}^{(2)}\,  E_{4,1}\,                        E_{4,3}.
$$

We say that a composition $\lm=(\lm_1,\ldots,\lm_n)$
is {\it coherent} with a weight $c_1\omega_1+\cdots+c_{n-1}\omega_{n-1}$ if
$c_i=\lm_i-\lm_{i+1}$ for any $i=1,\ldots,n-1$.

\begin{proposition}[\cite{Carter_Lusztig}]\label{basis of Weyl module}
Let $\omega$ be a dominant weight
and $\lm$ be any
partition coherent with $\omega$.
Then the vectors $F_Te^+_\omega$, where $T$ is a standard $\lm$-tableau,
form a basis of $\Delta(\omega)$.
\end{proposition}

To deal with
compositions,
it is convenient to introduce the
following sequences of length $n$: $\epsilon_i:=(0,\ldots,0,1,0,\ldots,0)$
with $1$ at position $i$ and $\alpha(i,j)=\epsilon_i-\epsilon_j$.
Here and in what follows, the sequences of $\Z^n$ are added
and subtracted componentwise.


For the rest of Sections~\ref{Tableaux} -- \ref{Comparison of flows},
we fix integers $a_1,\ldots,a_q,b_1,\ldots,b_q,i$ such that
\begin{equation}\label{equation:lcwman:2}
1\le a_1<a_2<\cdots<a_q\le i<b_q<\cdots<b_2<b_1\le n.
\end{equation}

Let $T$ be a regular row standard $\lm$-tableau and
$\Gamma$ be a flow of $\F_i(a_1,{\ldots},a_q;$\linebreak
$b_1,{\ldots},b_q)$.
We want to construct the regular row standard tableau $\sigma_{\Gamma,i}(T)$\label{sigma} so that
{
\renewcommand{\labelenumi}{{\bf\theenumi}}
\renewcommand{\theenumi}{($\sigma$-\arabic{enumi})}
\begin{enumerate}
    \item\label{property:sigma-1} $\N_T-\Gamma=\N_{\sigma_{\Gamma,i}(T)}$;
    \item\label{property:sigma-2} $\sigma_{\Gamma,i}(T)$ is a $(\lm_1,\ldots,\lm_i,\lm_{i+1}+1,\ldots,\lm_n+1)$-tableau.
\end{enumerate}}
\noindent
We obtain $\sigma_{\Gamma,i}(T)$ from $T$ as follows
\begin{enumerate}
\item\label{transformation:1}
      for every edge $(s,t)$ of $\Gamma$, we replace one $t$
      with $s$ in row $s$;
\item\label{transformation:2}
      to every row $k=i+1,\ldots,n$, we add one $k$;
\item\label{transformation:3} in the resulting tableau,
      we order the elements in rows to obtain a row standard tableau
      (automatically regular).
\end{enumerate}
One can see that~\ref{transformation:1}
ensures~\ref{property:sigma-1} and~\ref{transformation:2} ensures~\ref{property:sigma-2},
while~\ref{transformation:3} makes the tableau row standard and eliminates the uncertainty
stemming from the choice of the elements being removed and added in
steps~\ref{transformation:1} and~\ref{transformation:2}.

Obviously, it is not always possible to executed all replacements in
step~\ref{transformation:1}, in which case $\sigma_{\Gamma,i}(T)$ is not well-defined.
This happens if and only if for some edge $(s,t)$ of $\Gamma$,
row $s$ of $T$ does not contain $t$ as an entry, i.e. if and only if
$\Gamma_{s,t}=1$ and $(\N_T)_{s,t}=0$ for some $s$ and $t$.
Thus we have proved the following property:

{
\renewcommand{\labelenumi}{{\bf\theenumi}}
\renewcommand{\theenumi}{($\sigma$-\arabic{enumi})}
\begin{enumerate}
\setcounter{enumi}{2}
\item\label{property:sigma-3}
      $\sigma_{\Gamma,i}(T)$ is well-defined if and only if
      all entries of $\N_T-\Gamma$ are nonnegative.
\end{enumerate}}


For example, take $n=4$, $i=2$,
$\Gamma\in\F_2(1,2;4,3)$
 with edges $(1,4)$, $(2,3)$
and $T$ as in the above example. Then

{
\begin{center}
\setlength{\unitlength}{0.4mm}
\begin{picture}(150,0)
\put(0,0){\line(1,0){50}}
\put(0,-10){\line(1,0){50}}
\put(0,-20){\line(1,0){30}}
\put(0,-30){\line(1,0){20}}
\put(0,0){\line(0,-1){30}}
\put(10,0){\line(0,-1){30}}
\put(20,0){\line(0,-1){30}}
\put(30,0){\line(0,-1){20}}
\put(40,0){\line(0,-1){10}}
\put(50,0){\line(0,-1){10}}
\put(-25,-17){$T=$}
\put(3,-8.5){1} \put(13,-8.5){2} \put(23,-8.5){2} \put(33,-8.5){3}\put(43,-8.5){4}
\put(3,-18.5){2}\put(13,-18.5){3}\put(23,-18.5){3}
\put(3,-28.5){3}\put(13,-28.5){4}
\put(53,-15){\vector(1,0){37}}
\put(59,-12){\footnotesize step (i)}
\end{picture}

\vspace{-13pt}

\begin{picture}(-38,0)
\put(0,0){\line(1,0){50}}
\put(0,-10){\line(1,0){50}}
\put(0,-20){\line(1,0){30}}
\put(0,-30){\line(1,0){20}}
\put(0,0){\line(0,-1){30}}
\put(10,0){\line(0,-1){30}}
\put(20,0){\line(0,-1){30}}
\put(30,0){\line(0,-1){20}}
\put(40,0){\line(0,-1){10}}
\put(50,0){\line(0,-1){10}}
\put(3,-8.5){1} \put(13,-8.5){2} \put(23,-8.5){2} \put(33,-8.5){3}\put(43,-8.5){1}
\put(3,-18.5){2}\put(13,-18.5){2}\put(23,-18.5){3}
\put(3,-28.5){3}\put(13,-28.5){4}
\put(53,-15){\vector(1,0){34}}
\put(57,-12){\footnotesize step (ii)}
\end{picture}
\end{center}
%
%

\vspace{33pt}

\begin{center}
\setlength{\unitlength}{0.4mm}

\begin{picture}(150,0)
\put(0,0){\line(1,0){50}}
\put(0,-10){\line(1,0){50}}
\put(0,-20){\line(1,0){30}}
\put(0,-30){\line(1,0){30}}
\put(0,-40){\line(1,0){10}}
\put(0,0){\line(0,-1){40}}
\put(10,0){\line(0,-1){40}}
\put(20,0){\line(0,-1){30}}
\put(30,0){\line(0,-1){30}}
\put(40,0){\line(0,-1){10}}
\put(50,0){\line(0,-1){10}}
\put(3,-8.5){1} \put(13,-8.5){2} \put(23,-8.5){2} \put(33,-8.5){3}\put(43,-8.5){1}
\put(3,-18.5){2}\put(13,-18.5){2}\put(23,-18.5){3}
\put(3,-28.5){3}\put(13,-28.5){4}\put(23,-28.5){3}
\put(3,-38.5){4}
\put(53,-20){\vector(1,0){37}}
\put(57,-17){\footnotesize step (iii)}
\end{picture}

\vspace{-13pt}

\begin{picture}(-38,0)
\put(0,0){\line(1,0){50}}
\put(0,-10){\line(1,0){50}}
\put(0,-20){\line(1,0){30}}
\put(0,-30){\line(1,0){30}}
\put(0,-40){\line(1,0){10}}
\put(0,0){\line(0,-1){40}}
\put(10,0){\line(0,-1){40}}
\put(20,0){\line(0,-1){30}}
\put(30,0){\line(0,-1){30}}
\put(40,0){\line(0,-1){10}}
\put(50,0){\line(0,-1){10}}
\put(58,-17){$=\sigma_{\Gamma,i}(T)$}
\put(3,-8.5){1} \put(13,-8.5){1} \put(23,-8.5){2} \put(33,-8.5){2}\put(43,-8.5){3}
\put(3,-18.5){2}\put(13,-18.5){2}\put(23,-18.5){3}
\put(3,-28.5){3}\put(13,-28.5){3}\put(23,-28.5){4}
\put(3,-38.5){4}
\end{picture}
\end{center}
}

\vspace{2cm}

\noindent On the other hand, if we take $\Gamma'\in\F_2(2;4)$ with edge $(2,4)$,
then $\sigma_{\Gamma',i}(T)$ is not well-defined, since there is no $4$ in row $2$ of $T$
and step~\ref{transformation:1} can not be executed. We also have $(\N_T-\Gamma')_{2,4}=-1$
in accordance with~\ref{property:sigma-3}.


\begin{corollary}\label{corollary:lcwman:4}${}$\hspace{-5pt}
Let $T$ be a regular row standard $\lm$-tableau.
Then
$\xi_i(a_1,{\ldots},a_q;$\linebreak $b_1,\ldots,b_q)(F_T)$ is the sum
of $\sgn_i(\Gamma)F_{\sigma_{\Gamma,i}(T)}$ over all
flows $\Gamma\in\F_i(a_1,{\ldots},a_q;$\linebreak
$b_1,{\ldots},b_q)$ such that $\sigma_{\Gamma,i}(T)$ is well-defined.
\end{corollary}
\begin{proof}
By Lemma~\ref{lemma:lcwman:4}, we have
(recall Definition~\ref{definition:xi})
$$
\begin{array}{l}
\displaystyle
\xi_i(a_1,\ldots,a_q;b_1,\ldots,b_q)(F_T)=\\[6pt]
\displaystyle
\xi_i(a_1,\ldots,a_q;b_1,\ldots,b_q)\(\sum\nolimits_{N\in\UT(n)}F^{(N)}\delta_{N=\N_T}\)=\\[20pt]
\displaystyle
\sum\nolimits_{N\in\UT(n)}F^{(N)}
\sum\nolimits_{\Gamma\in\F_i(a_1,\ldots,a_q;b_1,\ldots,b_q)}
\sgn_i(\Gamma)\delta_{N+\Gamma=\N_T}=\\[20pt]
\displaystyle
\sum\nolimits_{\Gamma\in\F_i(a_1,\ldots,a_q;b_1,\ldots,b_q)}
\sgn_i(\Gamma)F^{(\N_T-\Gamma)}.
\end{array}
$$
Clearly, $F^{(\N_T-\Gamma)}=0$ if at least one entry of $\N_T-\Gamma$
is negative. Therefore, by~\ref{property:sigma-1} and~\ref{property:sigma-3},
we obtain
$$
\begin{array}{l}
\xi_i(a_1,\ldots,a_q;b_1,\ldots,b_q)(F_T)=\\[10pt]
\displaystyle\sum\Bigl\{\sgn_i(\Gamma)F^{(\N_T-\Gamma)}\|\Gamma\in\F_i(a_1,\ldots,a_q;b_1,\ldots,b_q)\,\&\\[6pt]
\;\;\quad\quad\quad\quad\quad\quad\quad\quad\quad\;\;\mbox{ all entries of }\N_T-\Gamma\mbox{ are nonnegative}\Bigr\}=\\[6pt]
\displaystyle\sum\Bigl\{\sgn_i(\Gamma)F_{\sigma_{\Gamma,i}(T)}\|\Gamma\,{\in}\,\F_i(a_1,{\ldots},a_q;b_1,{\ldots},b_q)\,\&\,
\sigma_{\Gamma,i}(T)\!\mbox{ is well-defined}\Bigr\}.
\!\!\!\!
\end{array}
$$

\end{proof}

\subsection{Comparison of tableaux}\label{Comparison of tableaux}
We make the following convention on orders. The symbol $\le$ (symbol $<$)
will always denote a nonstrict (resp. strict) partial order on a set $S$,
i.e. a transitive binary relation of $S$
such that $x\le x$ and $x\le y\,\&\,y\le x\,{\Rightarrow}\,x=y$
for any $x,y\in S$ (resp. $x<y\,\&\,y<x$ for no $x,y\in S$).
We also assume that for any $x,y\in S$, there holds
$$
{\arraycolsep=1pt
\begin{array}{rcll}
x\le y\Leftrightarrow x<y&\mbox{ or }&x=y,    &\quad\quad\quad\quad x\le y\Leftrightarrow y\ge x, \\[6pt]
x<y\Leftrightarrow x\le y&\&         &x\ne y, &\quad\quad\quad\quad x<y\Leftrightarrow y>x.
\end{array}}
$$
Consider two sequences $x=(x_1,\ldots,x_m)$ and $y=(y_1,\ldots,y_m)$
of elements of $S$.
We say that
$x$ {\it is smaller than} $y$
{\it in the antilexicographic
order} if there exists $i=1,\ldots,m$ such that $x_i<y_i$ and $x_j=y_j$ for $i<j\le m$.

The set of all compositions of length $n$ is endowed with
the following partial order called {\it dominance order}.
For two compositions $\lm{=}(\lm_1,{\ldots},\lm_n)$
and $\mu{=}(\mu_1,{\ldots},\mu_n)$, we write $\lm\le\mu$ if
$\sum_{j=1}^i\lm_j{\le}\sum_{j=1}^i\mu_j$ for any $i{=}1,{\ldots},n$.

Let $\mu$ be a composition of length $n$ and $Q$ be a $\mu$-tableau.
Then we set $\shape(Q):=\mu$\label{shape}. Given a composition $\lm$ of length $n$,
a row standard $\lm$-tableau $T$ and an integer $m=1,\ldots,n$,
we denote by $T[m]$\label{Tm} the tableau obtained from $T$ by removing
all nodes with entry greater than $m$. Clearly, $T[m]$ is a $\mu$-tableau
for some composition $\mu$ of length $n$. Therefore, we define
$$
\chain(T):=\bigl(\shape(T[1]),\shape(T[2]),\ldots,\shape(T[n])\bigr).\label{chain}
$$
For another row standard $\lm$-tableau $S$, we write $T<S$ if
$\chain(T)$ is smaller than $\chain(S)$ in the antilexicographic
order with dominance order on components.
We also refer the reader to~\cite[Section~2.2]{Kleshchev_gjs11}
for a description of this order.

It is convenient to define the sum of row standard tableaux $S$ and $T$
as the row standard tableau $S+T$, whose $i$th row is the result of gluing
and reordering the $i$th rows of $S$ and $T$.
Obviously, we have
\begin{equation}\label{equation:lcwman:2.5}
(S+T)[m]=S[m]+T[m],\quad\shape(S+T)=\shape(S)+\shape(T).
\end{equation}

{\it Example}. Let $n=5$. Then we have

{
\begin{center}
\setlength{\unitlength}{0.4mm}
\begin{picture}(45,0)
\put(0,0){\line(1,0){40}}
\put(0,-10){\line(1,0){40}}
\put(0,-20){\line(1,0){30}}
\put(0,-30){\line(1,0){10}}
\put(0,-40){\line(1,0){10}}
\put(0,0){\line(0,-1){40}}
\put(10,0){\line(0,-1){40}}
\put(20,0){\line(0,-1){20}}
\put(30,0){\line(0,-1){20}}
\put(40,0){\line(0,-1){10}}
\put(3,-8.5){1} \put(13,-8.5){2} \put(23,-8.5){4} \put(33,-8.5){5}
\put(3,-18.5){2}\put(13,-18.5){2}\put(23,-18.5){4}
\put(3,-28.5){5}
\put(3,-38.5){5}
\put(45,-22){$+$}
\end{picture}
\hspace{4mm}
\begin{picture}(45,0)
\put(0,0){\line(1,0){10}}
\put(0,-10){\line(1,0){20}}
\put(0,-20){\line(1,0){20}}
\put(0,-30){\line(1,0){20}}
\put(0,0){\line(0,-1){30}}
\put(10,0){\line(0,-1){30}}
\put(20,-10){\line(0,-1){20}}
\put(3,-8.5){3} 
\put(3,-18.5){3}\put(13,-18.5){5}
\put(3,-28.5){4} \put(13,-28.5){5}
\put(28,-22){$=$}
\end{picture}
\hspace{-4mm}
\begin{picture}(45,0)
\put(0,0){\line(1,0){50}}
\put(0,-10){\line(1,0){50}}
\put(0,-20){\line(1,0){50}}
\put(0,-30){\line(1,0){30}}
\put(0,-40){\line(1,0){10}}
\put(0,0){\line(0,-1){40}}
\put(10,0){\line(0,-1){40}}
\put(20,0){\line(0,-1){30}}
\put(30,0){\line(0,-1){30}}
\put(40,0){\line(0,-1){20}}
\put(50,0){\line(0,-1){20}}
\put(3,-8.5){1} \put(13,-8.5){2} \put(23,-8.5){3} \put(33,-8.5){4} \put(43,-8.5){5}
\put(3,-18.5){2}\put(13,-18.5){2}\put(23,-18.5){3}\put(33,-18.5){4}\put(43,-18.5){5}
\put(3,-28.5){4}\put(13,-28.5){5}\put(23,-28.5){5}
\put(3,-38.5){5}
\end{picture}
\end{center}
}

\vspace{1.7cm}
\noindent
Applying $\shape$ to each term gives (cf.~(\ref{equation:lcwman:2.5}))
$$
(4,3,1,1,0)+(1,2,2,0,0)=(5,5,3,1,0).
$$

\noindent

To understand how $\chain(T)$ changes when $T$
is replaced with $\sigma_{\Gamma,i}(T)$, we introduce the following
notation. For a flow $\Gamma$ and an integer $m$, we set
$$
\nu_m(\Gamma):=\sum\{\epsilon_s\|(s,t)\mbox{ is an edge of }\Gamma\mbox{ and }s\le m<t\}.\label{nu}
$$

\begin{lemma}\label{lemma:lcwman:5}
Let $\Gamma$ be a flow of
$\F_i(a_1,\ldots,a_q;b_1,\ldots,b_q)$ and $m=1,\ldots,n$.
Then we have
$$
\shape(\sigma_{\Gamma,i}(T)[m])=\shape(T[m])+\nu_m(\Gamma)+
\sum\nolimits_{k=i+1}^m\epsilon_k
$$
if $\sigma_{\Gamma,i}$ is applicable to $T$.
\end{lemma}
\begin{proof}
Let $S$ be the tableau obtained from $T$ in the same way as
$\sigma_{\Gamma,i}(T)$ but without step~\ref{transformation:2} and
$Q$ be the $(0^i,1^{n-i})$-tableau having one entry $k$
in each row $k=i+1,\ldots,n$. We clearly have
$\sigma_{\Gamma,i}(T)=S+Q$. It follows from~(\ref{equation:lcwman:2.5})
that $\shape(\sigma_{\Gamma,i}(T)[m])=\shape(S[m])+\shape(Q[m])$.
It is easy to see that $\shape(Q[m])=\sum_{k=i+1}^m\epsilon_k$.

Let us divide each row of $T$ and $S$ into the left part containing
entries $\le m$ and the right part containing entries $>m$.
Thus the right parts are deleted and the left parts remain
under $T\mapsto T[m]$ and $S\mapsto S[m]$.

Obviously, the left part of row s of $S$ has different length
than the left part of row s of $T$ if and only if $s\le m<t$
for some edge $(s,t)$ of $\Gamma$ (i.e., in step~\ref{transformation:1}, an entry of the
right part of row $s$ of $T$ is replaced by an entry that will go to the
left part of row $s$ of $S$).
In that case, the left part of row s of $S$
is longer by one than left part of row s of $T$,
which proves $\shape(S[m])=\shape(T[m])+\nu_m(\Gamma)$ and the lemma.
\end{proof}


\begin{proposition}[Straightening rule~\mbox{\cite[Theorem A.4]{Kleshchev_gjs11}}]\label{straightening rule}
If $\lm$ is a partition coherent with a dominant weight $\omega$
and $T$ is a regular row standard but not standard $\lm$-tableau,
then $F_Te^+_\omega$ can be rewritten as a linear combination
of vectors $F_Se^+_\omega$, where $S$ is a standard $\lm$-tableau
such that $S>T$.
\end{proposition}

%
%

\subsection{Comparison of flows}\label{Comparison of flows}
We are going to introduce a partial order on
$\F_i(a_1,{\ldots},a_q;b_1,{\ldots},b_q)$ and
show its relation to the order of Section~\ref{Comparison of tableaux}.
We start with an obvious observation.

\begin{proposition}\label{proposition:lcwman2:1}
Let $\Gamma$ be a nonempty flow with sources $\mathfrak a_1,\ldots,\mathfrak a_{\mathfrak q}$ and
sinks $\mathfrak b_1,\ldots,\mathfrak b_{\mathfrak q}$. Take the edge $(\hat s,\hat t)$
of $\Gamma$
such that $\hat t>t$ for any other edge $(s,t)$ of $\Gamma$.
We have $\hat t=\mathfrak b_j$
for some $j=1,\ldots,\mathfrak q$. The graph $\hat\Gamma$ obtained
from $\Gamma$ by removing $(\hat s,\hat t)$ is again
a flow with
{
\leftmargini=22pt
\begin{enumerate}
\item\label{nosource} sources $\mathfrak a_1,{\ldots},\mathfrak a_{\mathfrak q}$ and
      sinks $\mathfrak b_1,{\ldots},\mathfrak b_{j-1},\hat s,\mathfrak b_{j+1},{\ldots},\mathfrak b_{\mathfrak q}$
      if $\hat s\notin\{\mathfrak a_1,{\ldots},\mathfrak a_{\mathfrak q}\}${\rm;}
\item\label{source} sources $\mathfrak a_1,\ldots,\mathfrak a_{k-1},\mathfrak a_{k+1},\ldots,\mathfrak a_{\mathfrak q}$ and
      sinks $\mathfrak b_1,\ldots,\mathfrak b_{j-1},\mathfrak b_{j+1},\ldots,\mathfrak b_{\mathfrak q}$ \linebreak if
      $\hat s=\mathfrak a_k$. 
\end{enumerate}
}
\end{proposition}
Note that in this proposition the numbers $\mathfrak a_1, \ldots, \mathfrak a_{\mathfrak q}, \mathfrak b_1, \ldots, \mathfrak b_{\mathfrak q}$
do not have to satisfy~(\ref{equation:lcwman:2}) unlike $a_1, \ldots, a_q, b_1, \ldots, b_q$.

Let $\Gamma$ and $\Gamma'$ be distinct flows of
$\F_i(a_1,{\ldots},a_q;b_1,{\ldots},b_q)$ and
\begin{equation}\label{equation:lcwman:2.75}
(s_1,t_1),\ldots,(s_r,t_r);\quad
(s'_1,t'_1),\ldots,(s'_{r'},t'_{r'})
\end{equation}
be the edges of $\Gamma$ and $\Gamma'$ respectively ordered so that
$t_1>\cdots>t_r$ and $t'_1>\cdots>t'_{r'}$.

Proposition~\ref{proposition:lcwman2:1} explains why
one sequence of~(\ref{equation:lcwman:2.75}) can not be a proper beginning of
the other. Indeed, assume that $r<r'$ and $(s_k,t_k)=(s'_k,t'_k)$ for all
$k=1,\ldots,r$.
Let $\tilde\Gamma$ and $\tilde\Gamma'$ be the graphs obtained from
$\Gamma$ and $\Gamma'$ respectively by removing the edges
$(s_1,t_1),\ldots,(s_r,t_r)$. Then by Proposition~\ref{proposition:lcwman2:1} applied $r$ times,
$\tilde\Gamma$ and $\tilde\Gamma'$ are flows with the same sources and sinks.
To understand this, note that in~\ref{nosource} and~\ref{source} of
Proposition~\ref{proposition:lcwman2:1}, new sources and sinks depend
only on $a_1,\ldots,a_q,b_1,\ldots,b_q$ and the edge $(\hat s,\hat t)$
being removed. However $\tilde\Gamma$ is the empty flow, while
$\tilde\Gamma'$ is not, since $r'>r$. This is a contradiction.

By what we have just proved, there exists an index $l=1,\ldots,\min\{r,r'\}$
such that $(s_k,t_k)=(s'_k,t'_k)$ for $k<l$ and $(s_l,t_l)\ne(s'_l,t'_l)$.
Let $\Gamma_0$ and $\Gamma'_0$ be the graphs obtained form $\Gamma$ and $\Gamma'$
respectively by removing the edges $(s_1,t_1),\ldots,(s_{l-1},t_{l-1})$.
By Proposition~\ref{proposition:lcwman2:1} applied $l-1$ times,
$\Gamma_0$ and $\Gamma'_0$ are flows with the same sources and sinks.
Hence $t_l=t'_l$, since $t_l$ is the maximal sink of $\Gamma_0$ and
$t'_l$ is the maximal sink of $\Gamma'_0$. We write $\Gamma<\Gamma'$
if $s_l>s'_l$.

We still need to prove that the relation on
$\F_i(a_1,{\ldots},a_q;b_1,{\ldots},b_q)$ we have defined
is a linear order. Clearly, there exists a linear order $<$ on $\Z^2$
such that
$$
(s,t)<(s',t)\mbox{ if and only if }s>s'.
$$
Then it is obvious that $\Gamma<\Gamma'$ if and only if
$(s_1,t_1),\ldots,(s_r,t_r)$ is lexicographically smaller than
$(s'_1,t'_1),\ldots,(s'_{r'},t'_{r'})$ with respect to
the chosen order on $\Z^2$.

{\it Example.} Consider the flows $\Gamma$ and $\Gamma'$ of
$\F_4(1,3;6,5)$
with edges
$$
(3,6),(4,5),(2,4),(1,2);\quad (3,6),(4,5),(1,4)
$$
respectively. The first two edges are the same in both graphs.
Therefore looking at the third edges $(2,4)$ and $(1,4)$,
we conclude $\Gamma<\Gamma'$.

\begin{lemma}\label{lemma:lcwman:6}
Let $\Gamma$ and $\Gamma'$ be flows of
$\F_i(a_1,{\ldots},a_q;b_1,{\ldots},b_q)$ such that
$\Gamma<\Gamma'$. Then
$(\nu_1(\Gamma),\ldots,\nu_n(\Gamma))$
is smaller than $(\nu_1(\Gamma'),\ldots,\nu_n(\Gamma'))$
in the antilexicographic order with the dominance order on components.
\end{lemma}
\begin{proof} We use the notation introduced
before this lemma. Recall that we have
$(s_k,t_k)=(s'_k,t'_k)$ for $k<l$, $t_l=t'_l$ and $s_l>s'_l$.
Take any integer $m$ such that $t_l-1\le m\le n$.
We denote by $N_m$ and $N'_m$ the numbers of elements $>m$ in the sequences
$t_1,\ldots, t_r$ and $t'_1,\ldots, t'_{r'}$ respectively.
Since $t_{l+1}\le t_l-1\le m$ if $l<r$ and $t'_{l+1}\le t'_l-1\le m$ if $l<r'$, we obtain that
$N_m$ and $N'_m$ are the numbers of elements $>m$ in the sequences
$t_1,\ldots, t_l$ and $t'_1,\ldots, t'_l$ respectively. However these sequences are equal. Thus $N_m=N'_m$.

If $m\ge t_l$ then $N_m<l$ and
$\nu_m(\Gamma)=\nu_m(\Gamma')=\sum_{k=1}^{N_m}\delta_{s_k\le m}\;\epsilon_{s_k}$.
On the other hand, $N_{t_l-1}=l$ and
$$
\begin{array}{l}
\displaystyle
\nu_{t_l-1}(\Gamma')-\nu_{t_l-1}(\Gamma)=
\sum\nolimits_{k=1}^l\delta_{s'_k\le t_l-1}\;\epsilon_{s'_k}-\sum\nolimits_{k=1}^l\delta_{s_k\le t_l-1}\;\epsilon_{s_k}=\\[6pt]
\displaystyle
\epsilon_{s'_l}-\epsilon_{s_l}=\alpha(s'_l,s_l)>0.
\end{array}
$$
To check the last inequality, see the definition of the dominance order in
Section~\ref{Comparison of tableaux}.
\end{proof}

For
$\Gamma,\Gamma'\in\F_4(1,3;6,5)$
as in the above example, we have
$$
{
\arraycolsep=2pt
\begin{array}{rclrclrclrcl}
\nu_6(\Gamma) &=&0, &\quad\nu_5(\Gamma) &=&\epsilon_3, &\quad\nu_4(\Gamma) &=&\epsilon_3+\epsilon_4, &\quad\nu_3(\Gamma) &=&\epsilon_2+\epsilon_3,\\[3pt]
\nu_6(\Gamma')&=&0, &\quad\nu_5(\Gamma')&=&\epsilon_3, &\quad\nu_4(\Gamma')&=&\epsilon_3+\epsilon_4, &\quad\nu_3(\Gamma')&=&\epsilon_1+\epsilon_3,
\end{array}}
$$
and $\nu_3(\Gamma)<\nu_3(\Gamma')$ as predicted by Lemma~\ref{lemma:lcwman:6}.


\begin{lemma}\label{lemma:lcwman:7}
Let $T$ and $S$ be regular row standard tableaux, $\Gamma$ and $\Gamma'$
be flows of $\F_i(a_1,\ldots,a_q;b_1,\ldots,b_q)$ such that
$\Gamma\le\Gamma'$ and $\sigma_{\Gamma',i}(S)\le\sigma_{\Gamma,i}(T)$.
Then we have $S\le T$ and, moreover, $S<T$ if $\Gamma<\Gamma'$.
\end{lemma}
\begin{proof} By Lemma~\ref{lemma:lcwman:5}, for any $m=1,\ldots,n$, we have
$$
\shape(\sigma_{\Gamma,i}(T)[m])-\shape(\sigma_{\Gamma',i}(S)[m])=
\shape(T[m])-\shape(S[m])
$$
if $\Gamma=\Gamma'$. Therefore, $\sigma_{\Gamma',i}(S)\le\sigma_{\Gamma,i}(T)$
immediately implies $S\le T$.

Therefore, we consider only the case $\Gamma<\Gamma'$.
By Lemma~\ref{lemma:lcwman:6}, the sequence
$(\nu_1(\Gamma),\ldots,\nu_n(\Gamma))$ is smaller
than the sequence $(\nu_1(\Gamma'),\ldots,\nu_n(\Gamma'))$
in the antilexicographic order. This means that there
is some $x=1,\ldots,n$ such that $\nu_x(\Gamma)<\nu_x(\Gamma')$ and
$\nu_m(\Gamma)=\nu_m(\Gamma')$ for $x<m\le n$.

Let $y=1,\, \ldots,\, n$ be the number such that
$\shape(\sigma_{\Gamma',i}(S)[y])<$\linebreak $\shape(\sigma_{\Gamma,i}(T)[y])$ and
$\shape(\sigma_{\Gamma',i}(S)[m])=\shape(\sigma_{\Gamma,i}(T)[m])$ for $y<m\le n$
if $\sigma_{\Gamma',i}(S)<\sigma_{\Gamma,i}(T)$ and let $y:=-\infty$
if $\sigma_{\Gamma',i}(S)=\sigma_{\Gamma,i}(T)$.

{\it Case $y>x$.} For $m>y$, we have $m>x$ and thus
$\nu_m(\Gamma)=\nu_m(\Gamma')$. Therefore, by Lemma~\ref{lemma:lcwman:5},
we get
$$
0=\shape(\sigma_{\Gamma,i}(T)[m])-\shape(\sigma_{\Gamma',i}(S)[m])=
  \shape(T[m])-\shape(S[m]).
$$
On the other hand, $\nu_y(\Gamma)=\nu_y(\Gamma')$ and
Lemma~\ref{lemma:lcwman:5} yields
$$
0<\shape(\sigma_{\Gamma,i}(T)[y])-\shape(\sigma_{\Gamma',i}(S)[y])=
  \shape(T[y])-\shape(S[y]).
$$

{\it Case $y\le x$.} For $m>x$, we have $\nu_m(\Gamma)=\nu_m(\Gamma')$
and $m>y$. Therefore, by Lemma~\ref{lemma:lcwman:5},
we get
$$
0=\shape(\sigma_{\Gamma,i}(T)[m])-\shape(\sigma_{\Gamma',i}(S)[m])=
  \shape(T[m])-\shape(S[m]).
$$
On the other hand, Lemma~\ref{lemma:lcwman:5} yields
$$
\begin{array}{l}
0\le\shape(\sigma_{\Gamma,i}(T)[x])-\shape(\sigma_{\Gamma',i}(S)[x])=\\[6pt]
\shape(T[x])+\nu_x(\Gamma)-\shape(S[x])-\nu_x(\Gamma'),
\end{array}
$$
whence $\shape(T[x])-\shape(S[x])\ge\nu_x(\Gamma')-\nu_x(\Gamma)>0$.
\end{proof}

%
%

\begin{corollary}\label{corollary:lcwman:5}
Suppose that 
$F=\sum_{S\in X}\alpha_SF_S$, where $X$ is a nonempty
set of regular row standard $\lm$-tableaux and
$\alpha_S\in\K^*$. Let $T$ be a minimal element of $X$
and $\Gamma$ be the smallest element of
$\F_i(a_1,\ldots,a_q;b_1,\ldots,b_q)$.
If $\sigma_{\Gamma,i}(T)$ is well-defined, then
$$
\xi_i(a_1,\ldots,a_q;b_1,\ldots,b_q)(F)=
\sum\nolimits_{P\in Y}\beta_PF_P,
$$
where $Y$ is a
set of regular row standard
$(\lm_1,{\ldots},\lm_i,\lm_{i+1}{+}1,{\ldots},\lm_n{+}1)$-tab\-leaux,
$\beta_P\in\K^*$ and $\sigma_{\Gamma,i}(T)$ is a minimal element of $Y$.
\end{corollary}
\begin{proof} By Corollary~\ref{corollary:lcwman:4}, we obtain
$$
\begin{array}{l}
\displaystyle
\xi_i(a_1,\ldots,a_q;b_1,\ldots,b_q)(F)=
\sum\Bigl\{
\alpha_S
\sgn_i(\Gamma')F_{\sigma_{\Gamma',i}(S)}\|\\[6pt]
\displaystyle
\Gamma'\in\F_i(a_1,\ldots,a_q;b_1,\ldots,b_q),\;
S\in X\mbox{ and }\sigma_{\Gamma',i}(S)\mbox{ is well-defined}
\Bigr\}.
\end{array}
$$
Therefore, it suffices to show the following fact:
if $\Gamma'\in\F_i(a_1,\ldots,a_q;$\linebreak $b_1,\ldots,b_q)$,
$S\in X$, $\sigma_{\Gamma',i}(S)$ is well-defined and
$(\Gamma',S)\ne(\Gamma,T)$, then \linebreak
$\sigma_{\Gamma',i}(S)\not\le\sigma_{\Gamma,i}(T)$.

Indeed suppose that $\sigma_{\Gamma',i}(S)\le\sigma_{\Gamma,i}(T)$
for some $S$ and $\Gamma'$ satisfying the above hypothesis.
Recall that $\Gamma\le\Gamma'$, since $\Gamma$ is the smallest.
By Lemma~\ref{lemma:lcwman:7}, $S<T$
if $\Gamma<\Gamma'$, which contradicts the minimality of $T$.
If $\Gamma=\Gamma'$ then Lemma~\ref{lemma:lcwman:7} implies $S\le T$.
However, the case $S=T$ is impossible since $(\Gamma',S)\ne(\Gamma,T)$.
Hence, we again obtain the contradiction $S<T$.
%
\end{proof}

%
%

\subsection{Proof of Conjecture B}
Let $\omega=c_1\omega_1+\cdots+c_{n-1}\omega_{n-1}$ be a nonzero
dominant weight. Choose any $i=1,\ldots,n-1$ such that $c_i>0$.
Consider sequences
$i^{(q)}_1,{\ldots},i^{(q)}_{k_q}$; $i^{(q-1)}_1,\ldots,i^{(q-1)}_{k_{q-1}}$;
\ldots; $i^{(1)}_1,\ldots,i^{(1)}_{k_1}$ and integers
$a_1,\ldots,a_q,b_1,\ldots,b_q$
satisfying the hypothesis of Lemma~\ref{lemma:lcwman:4}.
For brevity, we shall use the abbreviation $X^V_\alpha:=X^V_{\alpha,1}$.
By Lemma~\ref{lemma:lcwman:0} and
 Definition~\ref{definition:xi},
for any $F\in\U^-$, we have
\begin{equation}\label{equation:lcwman:4}
{
\arraycolsep=1pt
\begin{array}{l}
\biggl[X^V_{\alpha_{i^{(q)}_1}},\ldots,X^V_{\alpha_{i^{(q)}_{k_q}}},X^V_{\alpha_i},
    X^V_{\alpha_{i^{(q-1)}_1}},\ldots,X^V_{\alpha_{i^{(q-1)}_{k_{q-1}}}},X^V_{\alpha_i},\ldots,\\[12pt]
\;\;X^V_{\alpha_{i^{(1)}_1}},\ldots,X^V_{\alpha_{i^{(1)}_{k_1}}},X^V_{\alpha_i},D^\omega_{\omega_i}\biggr](Fe^+_\omega)=\\[12pt]
\Bigl[\eta_{\alpha_{i^{(q)}_1}},\ldots,\eta_{\alpha_{i^{(q)}_{k_q}}},\eta_{\alpha_i},
      \eta_{\alpha_{i^{(q-1)}_1}},\ldots,\eta_{\alpha_{i^{(q-1)}_{k_{q-1}}}},\eta_{\alpha_i},\ldots,\\[12pt]
\;\;  \eta_{\alpha_{i^{(1)}_1}},\ldots,\eta_{\alpha_{i^{(1)}_{k_1}}},\eta_{\alpha_i},\theta_{\omega_i}
\Bigr](F)e^+_{\omega-\omega_i}=\\[12pt]
 \left\llbracket i^{(q)}_1,\ldots,i^{(q)}_{k_q},i,
 i^{(q-1)}_1,{\ldots},i^{(q-1)}_{k_{q-1}},i,\ldots,
 i^{(1)}_1,{\ldots},i^{(1)}_{k_1},i\right\rrbracket_i(F)\,e^+_{\omega-\omega_i}=\\[12pt]
\xi_i(a_1,\ldots,a_q;b_1,\ldots,b_q)(F)e^+_{\omega-\omega_i},
\end{array}}
\end{equation}
where $V=\Delta(\omega)\oplus\Delta(\omega-\omega_i)$ and
$D^{\,\omega}_{\omega_i}$ is any operator on $V$ whose restriction to
$\Delta(\omega)$ coincides with $d^{\,\omega}_{\omega_i}$.
Hence the restriction

\begin{equation}\label{equation:lcwman:4.5}
\begin{array}{l}
\biggl[
X^V_{\alpha_{i^{(q)}_1}},\ldots,X^V_{\alpha_{i^{(q)}_{k_q}}},X^V_{\alpha_i},
    X^V_{\alpha_{i^{(q-1)}_1}},\ldots,X^V_{\alpha_{i^{(q-1)}_{k_{q-1}}}},X^V_{\alpha_i},\ldots,\\[12pt]
\;\;X^V_{\alpha_{i^{(1)}_1}},\ldots,X^V_{\alpha_{i^{(1)}_{k_1}}},X^V_{\alpha_i},D^\omega_{\omega_i}
\biggr]\biggl|_{\Delta(\omega)}
\end{array}
\end{equation}
is a map from $\Delta(\omega)$ to $\Delta(\omega-\omega_i)$
depending only on $a_1,\ldots,a_q,b_1,\ldots,b_q,i$.
We denote it by $z^\omega_i(a_1,\ldots,a_q;b_1,\ldots,b_q)$.
Now we can rewrite~(\ref{equation:lcwman:4}) as
\begin{equation}\label{equation:lcwman:5}
z^\omega_i(a_1,\ldots,a_q;b_1,\ldots,b_q)(Fe^+_\omega)=
\xi_i(a_1,\ldots,a_q;b_1,\ldots,b_q)(F)e^+_{\omega-\omega_i}.
\end{equation}

\begin{lemma}\label{lemma:lcwman2:8}
Let $\omega=c_1\omega_1+\cdots+c_{n-1}\omega_{n-1}$ be a nonzero
dominant weight and $v$ be a nonzero vector of $\Delta(\omega)$.
Then there exist integers $a_1,\ldots,a_q$, \linebreak
$b_1,\ldots,b_q,i$ such that
$$
1\le a_1<a_2<\cdots<a_q\le i<b_q<\cdots<b_2<b_1\le n,\quad c_i>0
$$
and $z^\omega_i(a_1,\ldots,a_q;b_1,\ldots,b_q)(v)\ne0$.
\end{lemma}
\begin{proof} Let $\lm_j:=\sum_{k=j}^{n-1}c_j$ for $j=1,\ldots,n$.
Then $\lm:=(\lm_1,\ldots,\lm_n)$ is a partition coherent with $\omega$.
Then by Proposition~\ref{basis of Weyl module}
we can write $v=Fe^+_\omega$ with $F=\sum_{S\in X}\alpha_SF_S$,
where $X$ is a nonempty set consisting of standard $\lm$-tableaux
and $\alpha_S\in\K^*$. Let $T$ be a minimal element of $X$.
We denote by $i$ the maximal number of $1,\ldots,n-1$
such that $c_i>0$. Clearly, $i$ is the height of the first column
of $[\lm]$.

Let $i_0$ be the minimal number of $1,\ldots,i$ such that $i_0<T(i_0,1)$
or be $+\infty$ if there is no such number.
Obviously, $s=T(s,1)$ for any $s=1,\ldots,i_0-1$ and
$s<T(s,1)$ for any $s=i_0,\ldots,i$, since $T$ is standard.
Consider the graph $\Gamma$ with vertices $\Z$ and edges
$$
\bigl(i,T(i,1)\bigr),\;\bigl(i{-}1,T(i{-}1,1)\bigr),\;\ldots,\;\bigl(i_0,T(i_0,1)\bigr).
$$
The beginnings of all edges of $\Gamma$ are mutually distinct
and the ends of all edges of $\Gamma$ are mutually distinct too.
Therefore, $\Gamma$ is a flow (see Definition~\ref{definition:lcwman:1}).

Since the beginnings of the edges of $\Gamma$ are less than or equal to $i$,
the sources and transit points of $\Gamma$ are also less than or equal to $i$.
If $(s,T(s,1))$ is an edge of $\Gamma$ and $T(s,1){\le}i$,
then $i_0<T(s,1)\le i$ and $\Bigl(T(s,1),T\bigl(T(s,1),1\bigr)\Bigr)$
is an edge of $\Gamma$ beginning at $T(s,1)$.
Hence $T(s,1)$ is a transit point of $\Gamma$.
Therefore, all the sinks of $\Gamma$ are greater than $i$.

We denote by $a_1,\ldots,a_q$ the sources of $\Gamma$ and by $b_1,\ldots,b_q$
the sinks of $\Gamma$ ordered so that~(\ref{equation:lcwman:2}) holds.
Note that $\Gamma$ can be the empty flow, in which case $i_0=+\infty$
and $q=0$.

We are going to show that $\Gamma$ is the smallest flow of
$\F_i(a_1,{\ldots},a_q;b_1,{\ldots},b_q)$.
Take any flow $\Gamma'$ of $\F_i(a_1,{\ldots},a_q;b_1,{\ldots},b_q)$
distinct from $\Gamma$. In that case, we have $q>0$ and $1\le i_0\le i$.
Let $(s'_1,t'_1),\ldots,(s'_{r'},t'_{r'})$
be all the edges of $\Gamma'$ ordered so that $t'_1>\cdots>t'_{r'}$.
We set $r\,{:=}\,i\,{-}\,i_0\,{+}\,1$ and $s_k\,{:=}\,i\,{-}\,k\,{+}\,1$,\linebreak
$t_k:=T(i\,{-}\,k\,{+}\,1,1)$ for any $k:=1,\ldots,r$.
Then the edges of $\Gamma$ are \linebreak$(s_1,t_1),\ldots,(s_r,t_r)$ and
$t_1>\cdots>t_r$.

Following the comparison algorithm described in
Section~\ref{Comparison of flows}, we take
$l=1,\ldots,\min\{r,r'\}$ such that
$(s_k,t_k)=(s'_k,t'_k)$ for $k<l$ and $(s_l,t_l)\ne(s'_l,t'_l)$.
We have $t'_l=t_l$ and $s'_l\ne s_l$. Moreover, $s'_l\le i$ and
$s'_l\notin\{s'_1,\ldots,s'_{l-1},s_l\}=\{s_1,\ldots,s_l\}=\{i,\ldots,i-l+1\}$,
since $\Gamma'$ is a flow of $\F_i(a_1,\ldots,a_q;b_1,\ldots,b_q)$.
Hence $s'_l<i-l+1=s_l$ and $\Gamma'>\Gamma$.

By~(\ref{equation:lcwman:5}), we have
$$
z^\omega_i(a_1,\ldots,a_q;b_1,\ldots,b_q)(v)=
\xi_i(a_1,\ldots,a_q;b_1,\ldots,b_q)(F)e^+_{\omega-\omega_i}.
$$
It is easy to see that $\sigma_{\Gamma,i}(T)$ is well-defined and
is the tableau obtained from $T$ by replacing its first column by
the column of height $n$ having $k$ in each row $k=1,\ldots,n$.
Clearly, $\sigma_{\Gamma,i}(T)$ is standard.

By Corollary~\ref{corollary:lcwman:5}, we have
$$
\xi_i(a_1,\ldots,a_q;b_1,\ldots,b_q)(F)=
\sum\nolimits_{P\in Y}\beta_PF_P,
$$
where $Y$ is a
set of regular row standard
$(\lm_1,\ldots,\lm_i,\lm_{i+1}{+}1,{\ldots},\lm_n{+}1)$-tableaux,
$\beta_P\in\K^*$ and $\sigma_{\Gamma,i}(T)$ is a minimal element of $Y$.
Suppose that some $P\in Y$ is not standard.
Note that $(\lm_1,\ldots,\lm_i,\lm_{i+1}{+}1,{\ldots},\lm_n{+}1)$
is coherent with $\omega-\omega_i$.
Then by Proposition~\ref{straightening rule},
$F_Pe^+_{\omega-\omega_i}$ is a linear combination of
vectors $F_Qe^+_{\omega-\omega_i}$, where $Q$ is a standard
$(\lm_1,\ldots,\lm_i,\lm_{i+1}{+}1,{\ldots},\lm_n{+}1)$-tableau
such that $Q>P$. In particular, $Q\ne\sigma_{\Gamma,i}(T)$
by virtue of the minimality of $\sigma_{\Gamma,i}(T)$ in $Y$.
Hence the vector $z^\omega_i(a_1,\ldots,a_q;b_1,\ldots,b_q)(v)$
written in the standard basis of $\Delta(\omega-\omega_i)$
has a nonzero coefficient at $F_{\sigma_{\Gamma,i}(T)}e^+_{\omega-\omega_i}$
and thus is itself nonzero (Proposition~\ref{basis of Weyl module}).
\end{proof}

\begin{remark}
In this proof, it is important that all tableaux of $Y$ including
$\sigma_{\Gamma,i}(T)$ have shape coherent with $\omega-\omega_i$.
(Otherwise we would not be able to apply
Propositions~\ref{basis of Weyl module} and~\ref{straightening rule}).
We included step~\ref{transformation:2} in the definition of
$\sigma_{\Gamma,i}(T)$ exactly to ensure this property.
\end{remark}

\begin{theorem}\label{theorem:lcwman2:2}
If $G=A_{n-1}(\K)$, then Conjecture B holds for all $(F,\omega)$.
\end{theorem}
\begin{proof} We are going to apply Lemma~\ref{lemma:reformulation:B}.
Let $\omega$ be a nonzero dominant weight and
$v$ be a simply primitive vector
belonging to~(\ref{equation:lcwman:-1}).
We are going to apply the operator $z^\omega_i(a_1,\ldots,a_q;b_1,\ldots,b_q)$
to $v$, where $\omega-\omega_i$ is dominant and (\ref{equation:lcwman:2}) holds.
Recall that $z^\omega_i(a_1,\ldots,a_q;b_1,\ldots,b_q)$
equals~(\ref{equation:lcwman:4.5}), which in turn is a
linear combination with coeffitients $\pm1$ or $0$ of the restrictions $w|_{\Delta(\omega)}$,
where $w$ is a product of the operators
$$
\begin{array}{l}
X^V_{\alpha_{i^{(q)}_1}},\ldots,X^V_{\alpha_{i^{(q)}_{k_q}}},X^V_{\alpha_i},
    X^V_{\alpha_{i^{(q-1)}_1}},\ldots,X^V_{\alpha_{i^{(q-1)}_{k_{q-1}}}},X^V_{\alpha_i},\ldots,\\[12pt]
\;\;X^V_{\alpha_{i^{(1)}_1}},\ldots,X^V_{\alpha_{i^{(1)}_{k_1}}},X^V_{\alpha_i},D^\omega_{\omega_i}
\end{array}
$$
taken in an arbitrary order. If the last factor of $w$ is
$D^\omega_{\omega_i}$, then $w|_{\Delta(\omega)}(v)=0$, since $v$ belongs
to~(\ref{equation:lcwman:-1})
and the restriction of $D^\omega_{\omega_i}$ to $\Delta(\omega)$ is $d^\omega_{\omega_i}$.
Otherwise $w|_{\Delta(\omega)}(v)=0$,
since $v$ is simply primitive. Thus we have proved that
$z^\omega_i(a_1,\ldots,a_q;b_1,\ldots,b_q)(v)=0$.
Hence $v=0$ by Lemma~\ref{lemma:lcwman2:8}.
\end{proof}

\section{Appendix: List of Notations}\label{Appendix: List of Notations}

\tabcolsep=0pt
\begin{tabular}{p{2.6cm}p{9.5cm}}

$\Z^+$                   & set of nonnegative integers;\\[3pt]

$\U$                     & hyperalgebra of algebraic group $G$, p.~\pageref{hyperalgebra};\\[3pt]

$X_{\alpha,m}$,\, $H_{\alpha,m}$       & images of $X_\alpha^m/m!$ and $\binom{H_\alpha}m$ in $\U$ respectively;\\[3pt]

$\U^-$                   & subalgebra of $\U$ generated by all $X_{\alpha,m}$ with $\alpha<0$;\\[3pt]

$\U^0$                   & subalgebra of $\U$ generated by all $H_{\alpha,m}$;\\[3pt]

$\U^{-,0}$               &$\U^-\U^0$;\\[3pt]

$X_{\alpha,m}^V$         &operator on a rational $G$-module $V$ acting as the left\newline multiplication by $X_{\alpha,m}$;\\[3pt]

$X(T)$,\;$X^+(T)$        &sets of weights and dominant weights of torus $T$, p.~\pageref{XT};\\[3pt]

$L(\omega)$              & irreducible rational module
                             with highest weight $\omega$;\\[3pt]
$v^+_\omega$             & fixed nonzero vector of $L(\omega)$ having weight $\omega$;\\[3pt]

$\Delta(\omega)$         & Weyl module with highest weight $\omega$;\\[3pt]

$e^+_\omega$             & fixed nonzero vector of $\Delta(\omega)$ having weight $\omega$;\\[3pt]

$\nabla(\omega)$         & module contravariantly dual to $\Delta(\omega)$;\\[3pt]

$\h(\omega)$             &$a_1+\cdots+a_\ell$ for $\omega=a_1\omega_1+\cdots+a_\ell\omega_\ell$, p.~\pageref{h};\\[3pt]

$\ev^\omega$,\, $\r_{\alpha,m}^{\,\omega}$ &  p.~\pageref{r};\\[3pt]

$d^{\,\omega}_\delta$   & $\U^-$-homomorphism $\Delta(\omega)\to\Delta(\omega-\delta)$ that takes $e^+_\omega$ to $e^+_{\omega-\delta}$, Lemma~\ref{lemma:U^-}, p.~\pageref{lemma:U^-};\\[3pt]

$\theta_\delta$         &p.~\pageref{theta};\\[3pt]

$\eta_{\alpha,m}$       &p.~\pageref{eta};\\[3pt]

$S_q$                  &group of bijections of $\{1,\ldots,q\}$ (symmetric group);\\[3pt]

$\F_i(a_1,\ldots,b_q)$ &p. \pageref{F};\\[3pt]

$\llbracket i_1,\ldots,i_k\rrbracket_i$  &$[\eta_{\alpha_{i_1}},\ldots,\eta_{\alpha_{i_k}},\theta_{\omega_i}]$, p.~\pageref{ll};\\[3pt]

$\eta_\alpha$          &$\eta_{\alpha,1}$;\\[3pt]

$\delta_{\mathcal P}$  &$1$ if $\mathcal P$ is true and $0$ otherwise;\\[3pt]

$H_l$                  &$H_{\alpha_l,1}$;\\[3pt]

$\UT(n)$               & set of $n\times n$ 
                         matrices with entries in $\Z$
                         having $0$ on and under the main diagonal;\\[3pt]

$N^s$                  &$\sum_{b=1}^nN_{s,b}$, the sum of elements in row $s$ of $N\in\UT(n)$;\\[3pt]

$F^{(N)}$              &$\prod\nolimits_{1\le a<b\le n}E_{b,a}^{(N_{a,b})}$, p.~\pageref{FN};\\[3pt]

$\O(a,i,b)$            &Definition~\ref{definition:lcwman2:2}, p.~\pageref{definition:lcwman2:2};\\[3pt]

$\xi_i(a_1,\ldots,b_q)$&Definition~\ref{definition:xi}, p.~\pageref{definition:xi};\\[3pt]

$[\lm]$                &diagram of a composition $\lm$, p.~\pageref{diagram};\\[3pt]

$\N_T$                  &p.~\pageref{NT};\\[3pt]

$F_T$                  &$F^{(\N_T)}$;\\[3pt]

$\epsilon_i$           &$(0,\ldots,0,1,0,\ldots,0)$ of length $n$ with $1$ at position $i$;\\[3pt]

$\alpha(i,j)$          &$\epsilon_i-\epsilon_j$;\\[3pt]

\end{tabular}

\begin{tabular}{p{2.6cm}p{9.5cm}}

$\sigma_{\Gamma,i}(T)$ &p.~\pageref{sigma};\\[3pt]

$\shape(Q)$            &$\mu$ if $Q$ is a $\mu$-tableau, p.~\pageref{shape};\\[3pt]

$T[m]$                 &tableau obtained from $T$ by removing
                        all nodes with entry greater than $m$,
                        p.~\pageref{Tm};\\[3pt]

$\chain(T)$            &$\bigl(\shape(T[1]),\shape(T[2]),\ldots,\shape(T[n])\bigr)$, p.~\pageref{chain};\\[3pt]

$\nu_m(\Gamma)$        &$\sum\{\epsilon_s\|(s,t)\mbox{ is an edge of }\Gamma\mbox{ and }s\le m<t\}$, p.~\pageref{nu};\\[3pt]

$X^V_\alpha$           &$X^V_{\alpha,1}$;\\[3pt]

$z^\omega_i(a_1,\ldots,b_q)$ &p.~\pageref{equation:lcwman:5}, eq.~(\ref{equation:lcwman:5}).

\end{tabular}


\end{document}